\documentclass[review,hidelinks,onefignum,onetabnum]{siamart220329}

\usepackage{lipsum}
\usepackage{amsfonts}
\usepackage{graphicx}
\usepackage{epstopdf}
\usepackage{algorithmic}

\ifpdf
  \DeclareGraphicsExtensions{.eps,.pdf,.png,.jpg}
\else
  \DeclareGraphicsExtensions{.eps}
\fi

\newsiamremark{remark}{Remark}
\newsiamremark{hypothesis}{Hypothesis}
\crefname{hypothesis}{Hypothesis}{Hypotheses}
\newsiamthm{claim}{Claim}

\headers{Gauss-Newton Temporal Difference Learning}{Zhifa Ke, Junyu Zhang, Zaiwen Wen}

\title{Gauss-Newton Temporal Difference Learning with \\ Nonlinear Function Approximation}

\author{Zhifa Ke\thanks{Center for Data Science, Peking University, China, 
  (\email{kezhifa@stu.pku.edu.cn}).}
\and Junyu Zhang\thanks{National University of Singapore, Singapore, \email{junyuz@nus.edu.sg}.}
\and Zaiwen Wen\thanks{Beijing International Center for Mathematical Research, Peking University, China, (\email{wenzw@pku.edu.cn}).}}

\usepackage{mathtools}
\usepackage{tcolorbox}
\usepackage{nicefrac}
\usepackage{makecell}
\usepackage{enumitem}
\usepackage{amssymb}
\usepackage{booktabs}
\usepackage{subfigure}
\usepackage{textcomp}

\newtheorem{assumption}{Assumption}

\ifpdf
\hypersetup{
  pdftitle={Gauss-Newton Temporal Difference Learning with Nonlinear Function Approximation},
  pdfauthor={Zhifa Ke, Junyu Zhang, Zaiwen Wen}
}
\fi

\begin{document}

\maketitle

\begin{abstract}
In this paper, a Gauss-Newton Temporal Difference (GNTD) learning method is proposed to solve the Q-learning problem with nonlinear function approximation. In each iteration, our method takes one Gauss-Newton (GN) step to optimize a variant of Mean-Squared Bellman Error (MSBE), where target networks are adopted to avoid double sampling. Inexact GN steps are analyzed so that one can safely and efficiently compute the GN updates by cheap matrix iterations. Under mild conditions, non-asymptotic finite-sample convergence to the globally optimal Q function is derived for various nonlinear function approximations. In particular, for neural network parameterization with relu activation, GNTD achieves an improved sample complexity of  $\tilde{\mathcal{O}}(\varepsilon^{-1})$, as opposed to the $\mathcal{\mathcal{O}}(\varepsilon^{-2})$ sample complexity of the existing neural TD methods. An $\tilde{\mathcal{O}}(\varepsilon^{-1.5})$ sample complexity of GNTD is also established for general smooth function approximations. We validate our method via extensive experiments in several RL benchmarks, where GNTD exhibits both higher rewards and faster convergence than TD-type methods.
\end{abstract}



\section{Introduction}
\label{sect:intro}
Evaluating the Q function for certain policy $\pi$, and learning the optimal Q function, are important targets in Reinforcement Learning (RL). During the past decades, under the neural networks function approximation, deep Q-learning algorithms \cite{mnih2013playing/dqn, van2016deep/dqn} have achieved impressive empirical success in large-scale practical problems. They also serve as a fundamental building block for popular RL methods like TRPO \cite{schulman2015trust/trpo} and the actor-critic algorithms \cite{konda1999actor/ac,lillicrap2015continuous/ac,fujimoto2018addressing/ac}, 
where a proper parameterization of the Q function can be crucial to the scalability of the algorithm. Common examples include linear \cite{bhandari2018finite/nonasy, zou2019finite/nonasy, srikant2019finite/nonasy} and neural network \cite{cai2024neural/nonasy, xu2020finite/nonasy, agazzi2022temporal/nonasy} function approximations. In this work, we mainly study the Gauss-Newton-type semi-gradient methods for finding the evaluating Q functions for a fixed policy. The proposed method can be easily extended to optimal Q functions. 

Let $Q^\pi$ be the unknown state action value function to be estimated under policy $\pi$ and let $T^\pi$ be the corresponding Bellman operator. Then the standard formulation for policy evaluation with function approximation is to minimize the Mean-Squared Bellman Error (MSBE):
\begin{eqnarray}
\label{eq:MSBE}
    \min_{\theta} \mathcal{L}_\mu(\theta) & = &\|Q(\theta)-T^\pi Q(\theta)\|_\mu^2 \ = \ \mathbb{E}_{(s,a)\sim \mu}\big[\left(Q(s,a;\theta)-T^\pi Q(s,a;\theta)\right)^2\big],
\end{eqnarray} 
where $\mu$ is the \emph{stationary distribution} of the state-action pairs under the policy $\pi$. A direct optimization of MSBE can be very hard, as the double sampling issue denies an unbiased estimator of $\nabla \mathcal{L}_\mu(\theta)$. A practical technology to address the double sampling issue is to twist the loss function by introducing a target parameter $\theta_{targ}$ as follows
\begin{equation}
    \label{eq:targ-obj}
    \mathcal{L}_\mu(\theta;\theta_{targ}) = \|Q(\theta)-T^\pi Q(\theta_{targ})\|_\mu^2.
\end{equation} 
Based on this technique,
one popular class of methods for solving \eqref{eq:MSBE} are Temporal Difference (TD) \cite{sutton1988learning/ql} learning algorithm as well as its variants \cite{bradtke1996linear/lstd, sutton2009fast/gtd,sutton2009convergent/gtd, tu2018least/lstd}. Given the surrogate loss $\mathcal{L}_\mu(\theta;\theta_{targ})$ and the target parameter $\theta_{targ}^k=\theta^k$, the vanilla TD learning method updates by 
$$\theta^{k+1} = \theta^k - \eta g^k,$$
where $g^k$ is an unbiased estimator of $\nabla_{\theta}\mathcal{L}_\mu(\theta^k;\theta_{targ}^k)$. Since $\mathbb{E}[g^k]$ only contains a part of $\nabla_{\theta}\mathcal{L}_\mu(\theta^k)$, TD is also termed as the semi-gradient method. 
Though TD-type algorithms have been empirically successful in many numerical applications \cite{mnih2013playing/dqn,lillicrap2015continuous/ac,fujimoto2018addressing/ac, haarnoja2018soft/ac}, the  semi-gradient nature of the TD-type methods has mostly limited their theoretical convergence guarantees to linear function approximations \cite{sutton2009fast/gtd,sutton2009convergent/gtd}. 
Let us denote an $\varepsilon$-optimal solution to be a point $\bar{\theta}$ such that  $\mathbb{E}\left[\|Q(\bar\theta)-Q^\pi\|_\mu^2\right] \leq\varepsilon$ or 
$\|Q(\bar\theta)-Q^\pi\|_\mu^2 \leq\varepsilon$ with high probability\footnote{When non-squared error measure is adopted, to be consistent, we say $\bar\theta$ is $\varepsilon$-optimal if $\mathbb{E}\left[\|Q(\bar\theta)-Q^\pi\|_\mu\right] \leq\sqrt{\varepsilon}$ or 
$\|Q(\bar\theta)-Q^\pi\|_\mu \leq\sqrt{\varepsilon}$ with high probability.}. Then recent results on neural TD \cite{cai2024neural/nonasy,xu2020finite/nonasy} typically require a suboptimal $\mathcal{\mathcal{O}}(\varepsilon^{-2})$ sample complexity to find some $\varepsilon$-optimal solution. For general smooth function approximation, TD-type methods have divergence issues both in theory and practice \cite{tsitsklis/asy, maei2009convergent/gtd, achiam2019towards/var, brandfonbrener2019geo/nonasy}. 
Some variants, such as GTD2 and TDC \cite{maei2009convergent/gtd, wang2021non/smoothgtd, xu2021sample/smoothgtd}, consider the projected mean-squared Bellman error (MSPBE) in the parameter space and establish asymptotic or finite-time analyses. However, these algorithms can only prove that the parameters converge to a stationary point. 

Another popular approach for TD-learning is the Fitted Q-Iteration (FQI) method \cite{riedmiller2005neural/fqi, chen2019information/fqi, fan2020theoretical/fqi}, which repeatedly solves a nonlinear least square subproblem to fit the Q-iteration:
$$\theta^{k+1}\approx \mathrm{argmin}_\theta \, \mathcal{L}_\mu(\theta;\theta^k)$$
by sufficiently many stochastic gradient steps. Under expressive enough function approximation class such as over-parameterized neural networks, the FQI can largely approximate the contractive Q-iteration $Q(\theta^{k+1})\approx T^\pi Q(\theta^k)$ and hence retain desirable convergence properties. However, the need to solve the nonlinear least square subproblems may lead to an increased demand for RL state-action-transition samples and the per iteration computation time, which makes FQI less competitive against TD in practice.  

Interestingly, it can be observed that each iteration of TD method can be viewed as an inexact FQI that takes only one stochastic gradient descent step to optimize the nonlinear least square subproblem. The inexactness here can be an interpretation of the divergence behavior of TD learning for general function approximations. This motivates us to design the Gauss-Newton Temporal Difference (GNTD) learning algorithm, which takes only one Gauss-Newton step to optimize FQI's subproblem. The proposed method lies exactly in between FQI and TD, which requires fewer computation than FQI, while obtaining better theoretical convergence than TD. 

In this paper, we provide a complete finite-time convergence analysis of GNTD under neural network functions and general smooth functions under the i.i.d. sampling setting. Leveraging the powerful theoretical properties of GNTD in nonlinear function approximation, we further explore its practical performance in real-world problems. Specifically, for neural networks, we design a new algorithm framework based on the Kronecker-factored approximate curvature (K-FAC) method \cite{martens2015optimizing/kfac}. This algorithm demonstrates exceptional numerical performance across both continuous and discrete tasks, as well as offline and online settings.

\subsection{Contributions}
Our contributions are summarized as follows.
\begin{itemize}[leftmargin=*]
\item 
We propose the Gauss-Newton Temporal Difference (GNTD) learning algorithm to solve the Q learning problem with function approximation. We derive the convergence and sample complexity of our method under the i.i.d. sampling setting. Compared with the existing results of TD methods, GNTD achieves better sample complexities under nonlinear function approximations. For neural network approximation, GNTD achieves an improved $\tilde{\mathcal{O}}(\varepsilon^{-1})$ sample complexity as opposed the $\mathcal{O}(\varepsilon^{-2})$ sample complexity of the existing neural TD methods. Furthermore, as a by-product of the algorithm analysis, we also establish an $\tilde{\mathcal{O}}(\varepsilon^{-1.5})$ sample complexity of GNTD to find the global optima for smooth function approximation. 

\item We design an an alternative implementation of the Gauss-Newton step based on the Levenberg-Marquardt method and K-FAC \cite{martens2015optimizing/kfac} approximation method, named GNTD-KFAC, which proves to be highly effective in practice. We conduct extensive experiments to numerically validate the efficiency of GNTD. For both online and offline RL tasks, our method outperforms TD and its variants, including Deep Q-Network (DQN) method. 
\end{itemize}

\subsection{Related Work}
TD learning was first proposed for policy evaluation and Q-learning \cite{sutton1988learning/ql}, and it later on developed numerous variants. TD learning with linear approximation has been extensively studied, including Gradient TD (GTD) \cite{sutton2009fast/gtd,sutton2009convergent/gtd}, Accelerated GTD \cite{pan2017accelerated/gntd, ghiassian2020gradient/gntd}, and
Least-squares TD (LSTD) \cite{bradtke1996linear/lstd, boyan2002technical/lstd, ghavamzadeh2010lstd, tu2018least/lstd}. Subsequently, GTD \cite{maei2009convergent/gtd} was extended to the case of nonlinear function approximation. Since deep learning has achieved great success in large-scale problems, several variants based on neural network approximations have also been proposed, including Deep Q-network (DQN) \cite{mnih2013playing/dqn}, Preconditioned Q Networks (PreQN) \cite{achiam2019towards/var}, and TD3 \cite{fujimoto2018addressing/ac}. 

For the sample complexities of nonlinear TD method and its variants, current results mainly focus on neural network function approximation and general smooth function approximations. For neural TD methods, the key observation is that wide over-parameterized neural networks are approximately linear under the Neural Tangent Kernel (NTK) regime \cite{du2018gradient/ntk, zhang2019fast/ntk, allen2019convergence/ntk}. However, most existing neural TD methods only achieve the sub-optimal $\mathcal{\mathcal{O}}(\varepsilon^{-2})$ sample complexities \cite{cai2024neural/nonasy, brandfonbrener2019geo/nonasy, xu2020finite/nonasy}. 
A recent improvement was discovered in \cite{sun2022finite}, which derived an $\mathcal{O}(\varepsilon^{-\frac{2}{2-\alpha}})$  sample complexity for neural TD with some $\alpha\in(0,1]$. However, only the $\mathcal{\mathcal{O}}(\varepsilon^{-2})$ complexity with $\alpha=1$ can be guaranteed theoretically. Any complexity with $\alpha<1$ requires additional assumptions on the sparsity of the semi-gradients, which is hard to justify for neural TD with Gaussian initialization and restricted update region. There have also been recent works \cite{xu2021sample/smoothgtd, wang2021non/smoothgtd} that provide analysis for variants of TD for smooth function approximation. However, they can only guarantee that the parameters converge to a stationary point and are difficult to implement in practical large-scale tasks.

\section{Gauss-Newton Temporal Difference (GNTD) Learning}
\label{sect:alg}

\subsection{Preliminaries}
We consider the infinite-horizon discounted Markov decision process (MDP) $\mathcal{M}=(\mathcal{S}, \mathcal{A}, \mathbb{P}, r, \gamma)$, with state space $\mathcal{S}$, action space $\mathcal{A}$, reward function $r: \mathcal{S}\times \mathcal{A}\rightarrow [-R_{\max}, R_{\max}]$, transition probability 
$\mathbb{P}$, and a discount factor $\gamma \in(0,1)$. Let policy $\pi$ be a mapping that returns a probability distribution $\pi(\cdot|s)$ over the action space $\mathcal{A}$, for any state $s \in \mathcal{S}$. Then the state-action value function (Q-function) under policy $\pi$ is  $$Q^{\pi}(s,a)\!:=\!\mathbb{E}_\pi\!\!\left[\sum_{t=0}^{\infty} \gamma^{t} \!\cdot r(s_t,a_t) | s_{0}=s, a_{0}=a\right]\!, \,\,\forall s,a.$$
For any mapping $Q: \mathcal{S}\times \mathcal{A} \rightarrow \mathbb{R}$, we denote $T^\pi$ as the Bellman operator: 
$$T^\pi Q(s,a):=r(s, a)+\gamma P^\pi Q(s,a), \,\,\forall s,a,$$
where
$P^{\pi} Q(s,a)\!=\!\mathbb{E}\left[Q(s^{\prime}\!, a^{\prime}) | s^{\prime} \!\!\sim\! \mathbb{P}(\cdot | s, a), a^{\prime} \!\!\sim\! \pi\left(\cdot | s^{\prime}\right)\right].$ 

\subsection{The GNTD Method}
\label{section:GNTD-method}
Let the action-value function be parameterized as $Q(s,a;\theta)$ and we view $Q(\theta)$ as an $|\mathcal{S}\|\mathcal{A}|\times 1$ column vector, with $(s,a)$ being a multi-index arranged in the lexicographical order.
Recall the nonlinear least square subproblem of FQI:  
\begin{equation*} 
    \min_{\theta} \, \|Q(\theta) - T^\pi Q(\theta)\|^2_\mu,
\end{equation*} 
Unlike FQI that solves this problem with sufficiently many SGD steps, and unlike TD that optimizes the loss with merely one SGD step, we would like to propose a method lying between FQI and TD methods. Denote $J_Q(\theta)$ as the Jacobian matrix of the parameterized Q-function $Q(\theta)$. Then the GNTD method linearizes the $Q(\theta)$ in the FQI subproblem and updates the iterates by
\begin{equation} 
\label{eq:population}
\!\!\!\!\begin{cases}
d^{k}_* = \mathrm{argmin}_d \, \|Q(\theta^k) \!+\! J_Q(\theta^k)d \!-\! T^\pi Q(\theta^k)\|^2_\mu,\\ 
  \theta^{k+1} = \theta^k + \beta\cdot d^k_*.
\end{cases} 
\end{equation} 
The intuition behind the GNTD update is very clear. If  $\text{Span}\{J_Q(\theta^k)\}$ spanned by the columns of the Jacobian is expressive enough s.t. 
$J_Q(\theta^k)d^k_* = T^\pi Q(\theta^k) - Q(\theta^k).$
Under proper conditions, informally, one also has
$$\|Q(\theta^{k}+\beta d^k_*) - Q(\theta^k)- \beta J_Q(\theta^{k}) d^k_*\|\leq \mathcal{O}(\beta^2).$$
Combining the above two inequalities yields 
$$Q(\theta^{k+1}) = \beta\cdot T^\pi Q(\theta^k) + (1-\beta)\cdot Q(\theta^k) + \mathcal{O}(\beta^2).$$
Then the contraction of Bellman operator further yields
$\|Q(\theta^K) - Q^\pi\|_\infty \leq (1-(1-\gamma)\beta)^K\|Q(\theta^0) - Q^\pi\|_\infty + \mathcal{O}\big(\frac{\beta}{1-\gamma}\big).$ Therefore, setting $\beta=\mathcal{O}(\epsilon)$ or adopting a diminishing sequence of $\beta^k\to0$ will provide the convergence of GNTD in this ideal situation.  

For each tuple $\xi\!=\!(s,a,r,s',a')\sim\mathcal{D}$ where $\mathcal{D}$ is the distribution where  $(s,a)\!\sim\!\mu$, $r=r(s,a)$, $s'\!\sim\! \mathbb{P}(\cdot|s,a)$, $a'\!\sim\!\pi(\cdot|s')$, we define the loss function 
$$\ell(\xi,\theta^k,d):=\big(\nabla Q(s,a;\theta^k)^{\top}d+\delta^k(\xi)\big)^2$$ 
where  $\delta^k(\xi)=Q(s,a;\theta^k)-\left(r(s, a)+\gamma Q\left(s^{\prime}, a^{\prime};\theta^k\right)\right)$ is the TD error w.r.t. the tuple $\xi$. By removing constant terms, the GNTD subproblem for $d^k$ in \eqref{eq:population} can be rewritten in a more sample-friendly form: 
\begin{equation}
\label{eq:population-direction}
    d^k_* = \mathrm{argmin}_{d}\, L(\theta^k,d) := \mathbb{E}_{\xi\sim\mathcal{D}}\left[\ell(\xi,\theta^k,d)\right]. 
\end{equation}
Define the curvature matrix $H$ and the semi-gradient $g$ as
\begin{equation}
\label{eq:H-and-g}
\begin{aligned}
&H(\theta^k):=\mathbb{E}_{(s,a)\sim \mu}\left[ \nabla Q(s,a;\theta^k) \nabla Q(s,a;\theta^k)^{\top} \right],\\
&g(\theta^k):=\mathbb{E}_{\xi\sim\mathcal{D}}\left[ \delta^k(\xi)\cdot\nabla Q(s,a;\theta^k)\right].
\end{aligned}
\end{equation}
Then \eqref{eq:population-direction} has a closed form solution $d^k \!=\! -[H(\theta^k)]^{-1}g(\theta^k)$. Note that the solution $d^k$ is the Gauss-Newton direction for solving the nonlinear system $Q(\theta)=T^\pi Q(\theta^k)$ \cite{wright1999numerical}, and the semi-gradient $g(\theta^k)$ is exactly the expected TD direction.  Hence we call our method the Gauss-Newton Temporal Difference (GNTD) method. 

Note that the population update \eqref{eq:population} is not practically implementable. We introduce an empirical version of \eqref{eq:population-direction}, with an additional quadratic damping term to improve robustness:
\begin{equation}
\label{eq:stochastic-direction}
    d^k=\arg\min_{d}\,\, \frac{1}{N}\sum_{\xi\in\mathcal{D}_k}\ell(\xi,\theta^k,d)+\frac{\omega}{2} \|d\|_2^2,
\end{equation}
where $\mathcal{D}_k=\{(s_i,a_i,r_i,s'_i,a'_i)\}_{i=1}^{N}$ is a  batch of $N$ samples from $\mathcal{D}$.
Then \eqref{eq:stochastic-direction} leads to the empirical GNTD update:
\begin{equation}
\label{eq:stochastic}
\begin{cases}
\hat{H}^k = \frac{1}{N}\sum_{\xi\in\mathcal{D}_k}\nabla Q(s,a;\theta^k) \nabla Q(s,a;\theta^k)^{\top},\\
\hat{g}^k = \frac{1}{N}\sum_{\xi\in\mathcal{D}_k} \delta^k(\xi)\cdot\nabla Q(s,a;\theta^k),\\
\theta^{k+1}=\theta^k - \beta \cdot (\hat{H}^k+\omega I)^{-1} \hat{g}^k,
\end{cases} 
\end{equation}
as detailed in Algorithm \ref{alg:GNTD}. For large-scale function approximation class such as over-parameterized neural networks, the matrix inversion in \eqref{eq:stochastic} can be expensive. In this case, we provide a practically efficient implementation of GNTD using the Kronecker-factored Approximate Curvature (K-FAC) method. Please see the details in Section \ref{sect:kfac}.

\begin{algorithm}[tb]
\caption{The GNTD Algorithm}
\label{alg:GNTD}
\begin{algorithmic}
\STATE \textbf{Input:} Distribution $\mathcal{D}$, $Q$ function parameters $\theta^0$, batch size $N$, damping rate $\omega$, learning rate $\beta$.
\STATE Initialize the Q network parameters $\theta^0$.
\FOR{$k=0,1,\cdots,K-1$}
\STATE Obtain a dataset $\mathcal{D}_k=\{(s_i,a_i,r_i,s'_i,a'_i)\}_{i=1}^{N}$ of batch size $N$ from data distribution $\mathcal{D}$.
\STATE Update $\theta^{k+1}$ by \eqref{eq:stochastic}.
\ENDFOR
\STATE \textbf{Output:} $\theta^K$.
\end{algorithmic}
\end{algorithm}

\section{Convergence Results of GNTD Learning}
\label{sect:convergence}
In this section, we provide a finite-time convergence analysis of stochastic GNTD learning method linear, under neural network and general smooth function approximations for Algorithm \ref{alg:GNTD}. The detailed proof can be found in Section \ref{sect:proof}. First of all, let us make some basic assumptions on the data distribution. 

\begin{assumption}
\label{assumption:stationary}
The data distribution $\mu$ is the stationary distribution over the state-action pairs under the policy $\pi$. 
\end{assumption}
We write $U\!:=\!\text{diag}(\mu)$ as an $|\mathcal{S}| |\mathcal{A}|$-dimensional diagonal matrix, whose $(s,a)$-th diagonal element is $\mu(s,a)$.
Denote $\langle f,g\rangle_\mu\!:=\!\sum_{s,a} \mu(s,a) f(s,a)g(s,a)$, and $\left\|f\right\|_\mu^2=\langle f,f\rangle_\mu$. 
Then Assumption \ref{assumption:stationary} indicates that $T^\pi$ is a contraction w.r.t. $\|\cdot\|_\mu$ \cite{tsitsklis/asy}, that is,
$$
\|T^\pi Q_1-T^\pi Q_2\|_\mu \leq \gamma \|Q_1-Q_2\|_\mu, \,\,\forall Q_1, Q_2\in\mathbb{R}^{|\mathcal{S}|\times|\mathcal{A}|}.
$$
For the ease of the notation, we regard $Q^\pi$ as a $|\mathcal{S}\|\mathcal{A}|\times 1$ column vector similar to $Q(\theta)$, satisfying the Bellman equation $Q^\pi =T^\pi Q^\pi$ in later discussion.

\subsection{Linear Approximation}
Consider the linear function approximation with $\theta\in\mathbb{R}^d$: $$Q(s,a;\theta)=\phi(s,a)^{\top}\theta, \quad \forall s,a.$$
Let $\Phi\in\mathbb{R}^{|\mathcal{S}\|\mathcal{A}|\times d}$ be the collection of all feature vectors, whose $(s,a)$-th row equals $\phi(s,a)^{\top}$. Without loss of generality, we assume $\|\phi(s,a)\|\leq 1$ for any $(s,a)$. Then the minimizer $\theta^*$ of the MSBE \eqref{eq:MSBE} satisfies the projected Bellman equation
$$
\Phi \theta^*=\Pi_\mu T \Phi\theta^*,
$$
where $\Pi_\mu$ is the projection to the subspace $\{\Phi x : x\in \mathbb{R}^d\}$ under the inner product $\langle\cdot,\cdot\rangle_\mu$. We denote $Q^*:=\Phi\theta^*$ as the optimal linear approximator. 

In the under-parameterized regime where $d<|\mathcal{S}\|\mathcal{A}|$, we make the following assumption for the feature covariance matrix, in accordance with \cite{bhandari2018finite/nonasy}. 
\begin{assumption}
\label{assumption:cov-min-eig}
Let $\Sigma:=\mathbb{E}_{(s,a)\sim\mu}\left[\phi(s,a)\phi(s,a)^{\top}\right]$ be the feature covariance matrix. We assume $\Sigma\succ0$ and denote $\lambda_0>0$ as its minimum eigenvalue.  
\end{assumption}
Then the next theorem provides the convergence rate of stochastic linear GNTD in the under-parameterized regime, which matches the standard rate of linear TD under the same assumption \cite{bhandari2018finite/nonasy}.

\begin{theorem}
\label{theorem:sto-lin-under-parameterization}
Suppose Assumptions \ref{assumption:stationary} and \ref{assumption:cov-min-eig} hold. For any $\varepsilon\ll \|Q^0-Q^*\|_\mu$, if we set $\beta=\frac{(1-\gamma)\lambda_0}{4}$, the damping rate $\omega\in(0,1)$  for each iteration, then the output $\theta^K$ of Algorithm \ref{alg:GNTD} satisfies
$$
\begin{aligned}
    \|Q(\theta^K)-Q^\pi \|_\mu\ \leq\ &\left(1-\frac{(1-\gamma)^2\lambda_0}{8}\right)^{K}\|Q(\theta^0)-Q^*\|_\mu\\
    &+C_1\left(\omega+\sqrt{\frac{\log (K/\delta)}{N}}\right)+\frac{1}{1-\gamma} \|\Pi_\mu Q^\pi-Q^\pi\|_\mu.
\end{aligned}
$$
w.p. $1-\delta$, where $C_1>0$ is a given constant. 
\end{theorem}
If the true Q-function is realizable, namely,
\begin{equation}
    \label{eq:realizable}
    Q^\pi=Q^*\in\text{Span}\{\Phi\},
\end{equation}
then the intrinsic error term $\|\Pi_\mu Q^\pi-Q^\pi\|_\mu$ in Theorem \ref{theorem:sto-lin-under-parameterization} diminishes and the estimated Q-functions $Q(\theta^K)$ converges  linearly to the true Q-function $Q^\pi$. 

In accordance with the definition of the $\varepsilon$-optimal solution of \cite{bhandari2018finite/nonasy, cai2024neural/nonasy, xu2020finite/nonasy}, namely, $\mathbb{E}\left[\|Q(\theta^K)-Q^\pi\|_\mu^2\right]\leq\varepsilon$ or $\|Q(\theta^K)-Q^\pi\|_\mu^2\leq \varepsilon$ with high probability, we set the $\varepsilon$-optimal solution of GNTD to be $\|Q(\theta^K)-Q^\pi\|_\mu\leq \sqrt{\varepsilon}$. Suppose the Q-function is realizable, then setting $K=\mathcal{O}(\log \frac{1}{\varepsilon})$ in Theorem \ref{theorem:sto-lin-under-parameterization} yields $\|Q(\theta^K)-Q^* \|_\mu^2\leq\mathcal{O}(\varepsilon)$. Such a complexity $\tilde{\mathcal{O}}(\frac{1}{\varepsilon})$ complexity matches the result of stochastic linear TD \cite{bhandari2018finite/nonasy, zou2019finite/nonasy}.

\subsection{Neural Network Approximation}
For a two-layer neural network, we write:
\begin{equation}
\label{eq:neural-network}
    Q(s,a;\theta):=\frac{1}{\sqrt{m}}\sum_{r=1}^mb_r\psi(\theta_r^{\top}\phi(s,a)),
\end{equation}
where $\phi:\mathcal{S}\times\mathcal{A}\to\mathbb{R}^d$ is the feature mapping, $\theta\!:=\!(\theta_1^{\top},\cdots,\theta_m^{\top})^{\!\top}\!\in\mathbb{R}^{m\times d}, b:=(b_1,\cdots,b_m)^{\!\top}\in\mathbb{R}^m$ are the weight matrices, and $\psi(z):=\max\{0,z\}$ is the Relu activation. Similar to linear function approximation, we assume $\|\phi(s,a)\|\leq 1$ for any $(s,a)$. We denote $\theta^k_{r}$ the value of $\theta_r$ in iteration $k$. For each $r$, we initialize the weights $\theta^0_{r}\sim \mathcal{N}(0,\nu^2I)$ and $b_r\sim \rm{Unif} \{-1, +1\}$. The parameter $b$ will not be trained during the optimization. According to \cite{du2018gradient/ntk}, we make the following assumption to ensure the positive definiteness of the $\mu$-weighted Gram matrix.

\begin{assumption}
\label{assumption:feature-distance}
For all pairs $(s,a)\neq(s^\prime,a^\prime)$, we assume $\phi(s,a)\nparallel \phi(s^\prime,a^\prime)$. 
Moreover, We assume $\mu(s,a)>0$ to simplify the discussion. 
\end{assumption} 

The first setting in Assumption \ref{assumption:feature-distance} imply independence and boundedness between two feature vectors. Regarding the latter setting, let $\Omega:=\text{supp}(\mu)$ be the support of the distribution $\mu$. Then Assumption \ref{assumption:feature-distance} is equivalent to requiring that the feature vectors are linearly independent and the support $\Omega$ covers the full $\mathcal{S}\times\mathcal{A}$. In general, $\Omega=\mathcal{S}\times\mathcal{A}$ is a very strong assumption. However, it is not necessary for our analysis. For a policy $\pi$, under mild condition that the state transition Markov chain is aperiodic and irreducible, then $\mu(s,a)=0$ iff. $\pi(a|s)=0$. 
Then the Bellman equation is \emph{closed} on $\Omega$  since it does not involve $Q(s,a)$ for $(s,a)\notin\Omega$:
$$
Q^\pi (s,a)=r(s,a)+\!\!\sum_{(s^\prime,a^\prime)\in\Omega}\!\!\mathbb{P}(s^\prime|s,a)\cdot\pi(a^\prime|s^\prime)\cdot Q^\pi(s^\prime,a^\prime).
$$
This allows us to only care about the MSBE on $\Omega$:
$$
\min_\theta \mathbb{E}\left[\|Q(\theta)-T^\pi  Q(\theta)\|_\mu^2\right]=\mathbb{E}\left[\|Q(\theta)-T^\pi Q(\theta)\|_{\Omega,\mu}^2\right],
$$
where $||f||_{\Omega,\mu}^2\!:=\!\sum_{(s,a)\in\Omega}\mu(s,a)|f(s,a)|^2$. Therefore, Assumption \ref{assumption:feature-distance} can actually be relaxed to requiring the $\mu$-weighted Gram matrix on $\Omega$ to be positive definite. Though we will have no guarantee for $|Q(s,a;\theta)-Q^\pi(s,a)|$ for $(s,a)\notin\Omega$ in this case, it is not an issue. For example, when policy evaluation is applied as a built-in module of actor-critic methods, the Q-value outside $\Omega$ will never appear in the policy gradient formula. Therefore, we will assume $\Omega=\mathcal{S}\times\mathcal{A}$ for the ease of discussion.

\begin{theorem}
\label{theorem:sto-neural-netwrk}
Suppose Assumptions \ref{assumption:stationary} and \ref{assumption:feature-distance} hold. If we set $\beta,\omega\!\in\! (0,1)$ and the network width $m=\Omega \left(\frac{|\mathcal{S}|^3|\mathcal{A}|^3}{\delta^2}\right)$ for each iteration
, then the output $\theta^K$ of Algorithm \ref{alg:GNTD} satisfies
$$
\begin{aligned}
    \|Q(\theta^K)-Q^\pi\|_\mu\ \leq&\ (1-(1-\gamma)\beta)^K\|Q(\theta^0)-Q^\pi\|_\mu\\
    &\ +\frac{C_2}{(1-\gamma)^2} \left(\frac{\log(NK/\delta)}{\sqrt{N}}+\omega+m^{-\frac{1}{6}}\right)
\end{aligned}
$$
w.p. $1-\delta$, where 
$C_2>0$ is a given constant.  
\end{theorem}

Consistent with neural network approximation, we set the $\varepsilon$-optimal solution of GNTD to be $\|Q(\theta^K)-Q^\pi\|_\mu\leq \sqrt{\varepsilon}$. Then
Theorem \ref{theorem:sto-neural-netwrk} indicates 
that under proper choice of the parameters, our method finds an $\varepsilon$-optimal solution within a sample complexity of $\tilde{\mathcal{O}}\left(\varepsilon^{-1}\right)$, which is better than the $\mathcal{O}\left(\varepsilon^{-2}\right)$ sample complexity of neural TD \cite{cai2024neural/nonasy, xu2020finite/nonasy}. Our results are also better than the $\mathcal{O}\left(\varepsilon^{-\frac{2}{2-a}}\right), a\in(0,1]$ complexity of \cite{sun2022finite} while making no additional assumptions. 

\subsection{Smooth Function Approximation}
\label{sect:smooth-conv}
In this subsection, we consider the smooth function approximation that satisfies the following properties.
\begin{assumption}
\label{assumption:smooth}
    For all $(s,a)$, the function $Q(s,a;\cdot)$ is uniformly bounded by $M$, $L_1$-Lipschitz, and $L_2$-smooth: 
    $$\,\,\,\,\|Q(s,a;\theta_1)-Q(s,a;\theta_2)\|\leq L_1\|\theta_1-\theta_2\|,\qquad\, \forall \theta_1,\theta_2,$$
    $$\| \nabla Q(s,a;\theta_1)- \nabla Q(s,a;\theta_2)\|\leq L_2\|\theta_1-\theta_2\|, \quad \forall \theta_1,\theta_2.$$ 
\end{assumption}
Analogous to Assumption \ref{assumption:feature-distance}, we make the following requirement for the Jacobian matrix  $J_Q(\theta)$.
\begin{assumption}
\label{assumption:parameter-curve-min-eig}
$\exists\lambda_0>0$ such that for any $\theta$, we have $J_Q(\theta)^{\!\top}  U J_Q(\theta)\succeq \lambda_0 I$, where $U:=\text{diag}(\mu)$ is diagonal.
\end{assumption}

When $Q(\cdot)$ is linear, Assumption \ref{assumption:parameter-curve-min-eig} reduces to Assumption \ref{assumption:feature-distance}. It actually implies that the objective function $L(\theta^k,d)$ (cf. \eqref{eq:population-direction}) of the $d^k$ subproblem of the population update \eqref{eq:population} is $2\lambda_0$-strongly convex. Let us define the worst-case optimal fitting error of $L(\theta,d)$ over the parameter space $\mathcal{F}$ of $\theta$ as
\begin{equation}
    \label{eq:fitted-error}
    \varepsilon_{\mathcal{F}}^2=\max_{\theta\in \mathcal{F}}\min_{d} L(\theta,d)< \infty,
\end{equation}
then we have the following theorem.
\begin{theorem}
\label{theorem:sto-smooth}
Suppose Assumptions \ref{assumption:stationary}, \ref{assumption:smooth}, and \ref{assumption:parameter-curve-min-eig} hold. If we set $\beta,\omega\in (0,1)$,
then w.p. $1-\delta$, the output $\theta^K$ of Algorithm \ref{alg:GNTD} satisfies
$$
\begin{aligned}
    \|Q(\theta^K)-Q^\pi\|_\mu \leq&\  (1-(1-\gamma)\beta)^K\|Q(\theta^0)-Q^\pi\|_\mu \\ 
    &\ +\frac{C_3}{1-\gamma} \left(\omega+\sqrt{\frac{\log (2K/\delta)}{N}}+\varepsilon_{\mathcal{F}}+\beta\right)
\end{aligned}
$$
for a given constant $C_3>0$.  
\end{theorem}
By choosing the step size $\beta=\mathcal{O}(\varepsilon^{0.5})$ and the sample size $N=\mathcal{O}(\varepsilon^{-1})$, stochastic GNTD can output an 
$\varepsilon$-accurate Q-function approximator with  $\tilde{\mathcal{O}}(\varepsilon^{-0.5})$ iterations and in total $\tilde{\mathcal{O}}(\varepsilon^{-1.5})$ samples.
To our best knowledge, there is no explicit finite-time convergence analysis of TD or FQI for general smooth functions in terms of finding $\varepsilon$-optimal solutions. Existing results, such as GTD2 and TDC \cite{xu2021sample/smoothgtd, wang2021non/smoothgtd}, only analyze the complexity for finding $\varepsilon$-stationary points.

\section{Proof of Section \ref{sect:convergence}}
\label{sect:proof} 
In this section, we present detailed proofs of the finite-time analysis for GNTD under linear, neural and nonlinear smooth approximation.

\subsection{Proof of Theorem \ref{theorem:sto-lin-under-parameterization}}
\label{sect:lin-proof}
To simplify the notation, we rewrite $Q^k=Q(\theta^k)$ and $\delta^k=Q^k-T^\pi Q^k$. Before starting the analysis of stochastic GNTD with linear approximation, we introduce a few notation. In the $k$-th iteration, we obtain a batch of data tuples of the form $\xi=(s,a,r,s')$, we call this set of data tuples as $\mathcal{D}_k$. 
For each $\xi\in\mathcal{D}_k$, we define $\hat{g}(\theta^k,\xi)=\delta^k(\xi) \cdot \nabla Q(s,a;\theta^k)$. Consequently, the semi-gradient estimator $\hat{g}^k$ defined in \eqref{eq:stochastic} can also be written as $\hat{g}^k=\frac{1}{|\mathcal{D}_k|}\sum_{\xi\in\mathcal{D}_k}\hat{g}(\theta^k,\xi)$. Let $\hat{\mu}^k$ be the empirical estimator of $\mu$ based on the dataset $\mathcal{D}_k$, then the stochastic estimator  $\hat{H}^k$ defined in \eqref{eq:stochastic} can be equivalently written as $\hat{H}^k=\hat{H}(\theta^k)=(J^k)^\top \hat{U}^k J^k$, where $\hat{U}^k=\text{diag}(\hat{\mu}^k)$.


Recalling the population update \eqref{eq:population} and the stochastic update \eqref{eq:stochastic} of GNTD, We define $Q^{k}_*=Q(\theta^k_*)$. The following three lemmas employ $Q^k_*$ as an intermediary to derive an upper bound for $\|Q^k-Q^*\|_\mu$.

\begin{lemma}
\label{lemma:pop-semi-grad-bound}
(\cite{bhandari2018finite/nonasy}) Under Assumption \ref{assumption:stationary}, for given $k$, we have
$$
(\theta^*-\theta^k)^\top g(\theta^k)\leq -(1-\gamma)\|Q^*-Q^k\|_\mu^2\qquad\mbox{and}\qquad \|g(\theta^k)\|_2\leq (1+\gamma)\|Q^*-Q^k\|_\mu.
$$
\end{lemma}

\begin{lemma}
\label{lemma:pop-gap-1-step}
Under Assumption \ref{assumption:cov-min-eig}, for given $k$, we have
$$
\|Q^{k+1}_*-Q^*\|_\mu^2\leq  \left(1-2(1-\gamma)\beta+\frac{4\beta^2}{\lambda_0}\right)\|Q^k-Q^*\|_\mu^2.
$$
\end{lemma}

\begin{proof}
Recall the update formula from \eqref{eq:population}, then we have 
\begin{eqnarray*}
    \|Q^{k+1}_*-Q^*\|_\mu^2&=&\|\Phi \theta^{k+1}_*-\Phi \theta^*\|_\mu^2 \\
    &=&\|\Phi \theta^k-\Phi\theta^*\|_\mu^2+\beta^2\|\Phi d^k_*\|_\mu^2+2\beta \left<\Phi (\theta^k-\theta^*),\Phi d^k_*\right>_\mu \\
    &=& \|Q^k-Q^*\|_\mu^2+\beta^2g(\theta^k)^\top (\Phi^\top U\Phi)^{-1}\Phi^\top U\Phi(\Phi^\top U\Phi)^{-1}g(\theta^k)-2\beta (\theta^k-\theta^*)^\top g(\theta^k)  \\
    &=&\|Q^k-Q^*\|_\mu^2+\beta^2\|g(\theta^k)\|_{(\Phi^\top U\Phi)^{-1}}^2-2\beta (\theta^k-\theta^*)^\top g(\theta^k).
\end{eqnarray*}
By Assumption \ref{assumption:cov-min-eig} and Lemma \ref{lemma:pop-semi-grad-bound}, we have 
\begin{eqnarray*}
    \|Q^{k+1}-Q^*\|_\mu^2&\leq & \|Q^k-Q^*\|_\mu^2+\frac{\beta^2}{\lambda_0}\|g(\theta^k)\|^2-2\beta(1-\gamma)\|Q^k-Q^*\|^2_\mu \\
    &\leq & \left(1-2(1-\gamma)\beta+\frac{4\beta^2}{\lambda_0}\right)\|Q^k-Q^*\|_\mu^2.
\end{eqnarray*}
\end{proof}

\begin{lemma}
\label{lemma:sto-gap-1-step}
We define $\tau_1^2=9(1+\gamma^2)\|\theta^*\|^2+9R_{\max}^2+\frac{12(1+\gamma)^2\|Q^0-Q^*\|_\mu^2}{\lambda_0}$ and set $\varepsilon\ll \|Q^0-Q^*\|_\mu,\beta=\frac{(1-\gamma)\lambda_0}{4}$.
Then for any $\delta\in(0,1)$,
under Assumptions \ref{assumption:stationary} and \ref{assumption:cov-min-eig}, we have $\|Q^k-Q^*\|_\mu\leq \|Q^0-Q^*\|_\mu$. Furthermore, there exists a constant $C_1$ w.r.t. $\tau_1,\gamma,\lambda_0$ such that
$$
\|Q^k-Q^*\|_\mu \leq \left(1-\frac{(1-\gamma)^2\lambda_0}{8}\right)^{k+1}\|Q^0-Q^*\|_\mu+C_1\left(\omega+\sqrt{\frac{\log (K/\delta)}{N}}\right)
$$
w.p. $1-\delta$ for any $1\leq k\leq K$.
\end{lemma}

\begin{proof}
We prove this lemma by induction. Initially, it is evident for $k=0$. Subsequently, we consider the case where it holds for $k\leq t$. We then contemplate the case for $t+1$.  It is noteworthy that for any $\theta$ and tuple $\xi$, we have
\begin{eqnarray}
\label{eq:theta-0-bound}
\|\hat{g}(\theta,\xi)\|^2 &=& \|\phi(s,a)(Q(s,a;\theta)-r(s,a)-\gamma Q(s',a';\theta))\|^2 \nonumber\\
& \overset{(i)}{\leq} & 3\left(Q(s,a;\theta^*)-r(s,a)-\gamma Q(s',a';\theta^*)\right)^2+3\left(Q(s,a;\theta)-Q(s,a;\theta^*)\right)^2\nonumber\\
& & +3\gamma^2\left(Q(s',a';\theta)-Q(s',a';\theta^*)\right)^2 \\
& \overset{(ii)}{\leq} & 9(1+\gamma^2)\|\theta^*\|^2+9R_{\max}^2+3(1+\gamma^2)\|\theta-\theta^*\|^2 \nonumber\\
& \overset{(iii)}{\leq} & 9(1+\gamma^2)\|\theta^*\|^2+9R_{\max}^2+\frac{12(1+\gamma)^2\|Q(\theta)-Q^*\|_\mu^2}{\lambda_0},\nonumber
\end{eqnarray}
where (i) and (ii) are both due to the triangle inequality $\|x+y+z\|^2\leq 3(\|x\|^2+\|y\|^2+\|z\|^2)$, and (iii)is due to $\|\theta-\theta^*\|\leq \frac{1}{\sqrt{\lambda_0}}\|Q(\theta)-Q^*\|_\mu$. This means that 
$\sigma_k:=\mbox{Var}(\hat{g}(\theta^k,\xi)\mid \theta^k)\leq \max_\xi \|\hat{g}(\theta^k,\xi)\|^2\leq \tau_1^2$ (replacing $\theta$ with $\theta^k$ in \eqref{eq:theta-0-bound} when $0\leq k\leq t$).

Next, we compute 
\begin{eqnarray}
\label{eq:lin-d*-dk}
\|d^k-d^k_*\|&=& \left\|\left(\hat{H}^k+\omega I\right)^{-1}\hat{g}^k - \left(\Phi^\top U\Phi\right)^{-1}g^k\right\| \\
&\leq &\left\|\left(\hat{H}^k+\omega I\right)^{-1}\hat{g}^k - \left(\Phi^\top U\Phi\right)^{-1}\hat{g}^k\right\|+\left\| \left(\Phi^\top U\Phi\right)^{-1}(\hat{g}^k-g^k)\right\|. \nonumber
\end{eqnarray}
According to Assumption \ref{assumption:cov-min-eig} and Lemma \ref{lemma:matrix-bernstein}, if for any $\eta\in(0,1)$, sample size $|\mathcal{D}_k|\geq\frac{6}{\eta^2}\log \frac{2d}{\delta_0}$, we have
\begin{equation}
\label{eq:lin-under-gn-matrix-approx}
\lambda_{\min}(\Phi^{\top}\hat{U}^k\Phi)\geq \lambda_{\min}(\Phi^{\top}U \Phi)-\|\Phi^{\top}\hat{U}^k\Phi-\Phi^{\top} U \Phi\|_2\geq \lambda_0-\eta,\quad w.p.\ 1-\delta_0.
\end{equation}
Let $\eta=\omega$. Thus we have w.p. $1-\delta_0$ that 
\begin{eqnarray}
\label{eq:lin-under-matrix-1}
    &\ &\left\|(\hat{H}^k+\omega I)^{-1}-\left(\Phi^{\top} U \Phi\right)^{-1}\right\|_2 \nonumber\\
    &=&\Big\|\big(\Phi^{\top} U \Phi\big)^{-1}\big(\Phi^{\top}\hat{U}^k\Phi+\omega I-\Phi^{\top} U \Phi\big)\big(\Phi^{\top}\hat{U}^k\Phi+\omega I\big)^{-1}\Big\|_2  \\
    &\leq &\left\|\left(\Phi^{\top} U \Phi\right)^{-1}\right\|_2\cdot \left\|\Phi^{\top}\hat{U}^k\Phi+\omega I-\Phi^{\top} U \Phi\right\|_2\cdot\Big\|\left(\Phi^{\top}\hat{U}^k\Phi+\omega I\right)^{-1}\Big\|_2 \nonumber \\
    &\overset{(i)}{\leq} & \frac{1}{\lambda_0}\cdot(\eta+\omega)\cdot \frac{1}{\lambda_0+\omega-\eta} 
    = \frac{1}{\lambda_0^{2}}(\eta+\omega), \nonumber
\end{eqnarray}
where (i) is due to \eqref{eq:lin-under-gn-matrix-approx} and Assumption \ref{assumption:cov-min-eig}.
Therefore, for the first term of \eqref{eq:lin-d*-dk}, we have w.p. $1-\delta_0$ that 
\begin{eqnarray}
\label{eq:lin-err-1}
&\ &\left\|\left(\hat{H}^k+\omega I\right)^{-1}\hat{g}^k - \left(\Phi^\top U\Phi\right)^{-1}\hat{g}^k\right\|\\
&\leq& \left\|\left(\hat{H}^k+\omega I\right)^{-1} - \left(\Phi^\top U\Phi\right)^{-1}\right\|\cdot \|\hat{g}^k\| \leq \frac{2\tau_1\omega}{\lambda_0^{2}}, \nonumber
\end{eqnarray}

Next, by Lemma \ref{lemma:vector-bernstein}, we have
$$
\|g(\theta^k)-\hat{g}^k\|\leq \sigma_k\sqrt{\frac{3\log (1/\delta_0)}{|\mathcal{D}_k|}},\qquad w.p.\ 1-\delta_0.
$$
For the second term of \eqref{eq:lin-d*-dk}, we have
\begin{eqnarray}
\label{eq:lin-err-2}
\left\| \left(\Phi^\top U\Phi\right)^{-1}(\hat{g}^k-g^k)\right\|&\leq& \frac{\sigma_k}{\lambda_0}\cdot\sqrt{\frac{3\log (1/\delta_0)}{|\mathcal{D}_k|}}\leq \frac{\tau_1}{\lambda_0}\cdot\sqrt{\frac{3\log (1/\delta_0)}{|\mathcal{D}_k|}}
\end{eqnarray}
w.p. $1-\delta_0$. Plugging \eqref{eq:lin-err-1} and \eqref{eq:lin-err-2} into \eqref{eq:lin-d*-dk} yields that
$$
\|d^k-d^k_*\|\leq \frac{2\tau_1\omega}{\lambda_0^{2}}+\frac{\tau_1}{\lambda_0}\cdot\sqrt{\frac{3\log (1/\delta_0)}{|\mathcal{D}_k|}}.
$$

Now we choose $\beta=\frac{(1-\gamma)\lambda_0}{4}$. According to the estimation of $\|d^k-d^k_*\|$, we have
\begin{eqnarray*}
\|Q^{k+1}-Q^*\|_\mu&\leq& \|Q^{k+1}-Q^{k+1}_*\|_\mu-\|Q^{k+1}_*-Q^*\|_\mu\\
&\overset{(i)}{\leq}& \|d^k-d^k_*\|+\left(1-\frac{(1-\gamma)^2\lambda_0}{8}\right)\|Q^k-Q^*\|_\mu,
\end{eqnarray*}
where (i) is due to Lemma \ref{lemma:pop-gap-1-step}, the estimation of $\|d^k-d^k_*\|$ and the $Q$-function is 1-Lipschitz w.r.t. $\mu$-norm. 
Choose $\delta_0=\frac{\delta}{2K}$ and the sample size $|\mathcal{D}_k|=N\geq\frac{6}{\omega^2}\log \frac{2d}{\delta_0}$, according to the inductive hypothesis and setting $k=t$, we have
\begin{eqnarray*}
\|Q^{t+1}-Q^*\|_\mu&\leq& \|d^t-d^t_*\|+\left(1-\frac{(1-\gamma)^2\lambda_0}{8}\right)\|Q^t-Q^*\|_\mu \\ 
&\overset{(i)}{\leq}& \frac{2\tau_1\omega}{\lambda_0^{2}}+\frac{\tau_1}{\lambda_0}\cdot\sqrt{\frac{3\log (2K/\delta)}{N}}+\left(1-\frac{(1-\gamma)^2\lambda_0}{8}\right)\|Q^0-Q^*\|_\mu,
\end{eqnarray*}
where (i) is due to \eqref{eq:lin-err-1}.
Choose $\omega=\mathcal{O}(\varepsilon^{0.5}),N=(\varepsilon^{-1}),K=\mathcal{O}(\varepsilon^{-0.5})$, and let $\varepsilon$ be sufficiently small such that 
$$
\frac{2\tau_1\omega}{\lambda_0^{2}}+\frac{\tau_1}{\lambda_0}\cdot\sqrt{\frac{3\log (2K/\delta)}{N}}\leq \frac{(1-\gamma)^2\lambda_0}{16}\|Q^0-Q^*\|_\mu,
$$
which ensures $\|Q^{t+1}-Q^*\|_\mu\leq \|Q^0-Q^*\|_\mu$. Then there exists a constant $C_1$ such that 
\begin{eqnarray*}
\|Q^{t+1}-Q^*\|_\mu&\leq&  \frac{C_1(1-\gamma)^2\lambda_0}{8}\left(\omega+\sqrt{\frac{\log (K/\delta)}{N}}\right)+\left(1-\frac{(1-\gamma)^2\lambda_0}{8}\right)\|Q^0-Q^*\|_\mu\\
&\leq& \left(1-\frac{(1-\gamma)^2\lambda_0}{8}\right)^{t+1}\|Q^0-Q^*\|_\mu+C_1\left(\omega+\sqrt{\frac{\log (K/\delta)}{N}}\right).
\end{eqnarray*}
Hence, it is valid for $k=t+1$, which concludes our proof by induction.
\end{proof}

Prior to this, our focus was on estimating $\|Q^k-Q^*\|_\mu$. Now, we present the relationship between $Q^*$ and $Q^\pi$, thereby aiding the proof of Theorem \ref{theorem:sto-lin-under-parameterization}.

\begin{lemma}
\label{lemma:gap-between-Q*-Qpi} (\cite{cai2024neural/nonasy})
Under Assumption \ref{assumption:stationary}, 
$
\|Q^*-Q^\pi\|_\mu \leq \frac{1}{1-\gamma} \|\Pi_\mu Q^\pi-Q^\pi\|_\mu.
$
\end{lemma}

Now we are ready to prove Theorem \ref{theorem:sto-lin-under-parameterization}.

\begin{proof}
By Lemmas \ref{lemma:sto-gap-1-step} and \ref{lemma:gap-between-Q*-Qpi}, we have w.p. $1-\delta$ that 
\begin{eqnarray*}
    &\ &\|Q^K-Q^\pi\|_\mu \leq \|Q^K-Q^*\|_\mu+\|Q^*-Q^\pi\|_\mu \\
    &\leq& \left(1-\frac{(1-\gamma)^2\lambda_0}{8}\right)^{K}\|Q^0-Q^*\|_\mu+C_1\left(\omega+\sqrt{\frac{\log (K/\delta)}{N}}\right)+\frac{1}{1-\gamma} \|\Pi_\mu Q^\pi-Q^\pi\|_\mu.
\end{eqnarray*}
Thus we complete the proof. 
\end{proof}

\subsection{Proof of Theorem \ref{theorem:sto-neural-netwrk}}
\label{sect:nn-proof}
For neural network function approximation, we define the $\mu$-weighted Gram matrix $G^\infty \in \mathbb{R}^{|\mathcal{S}| |\mathcal{A}|\times |\mathcal{S}| |\mathcal{A}|}$. Let $\mu(s_i, a_i)$ be the $i$-th diagonal element of the diagonal matrix $U$. The $(i,j)$-th element of $G^\infty$ can be written as follows:
$$
\begin{aligned}
G^{\infty}_{ij}:=\mathbb{E}_{\theta\sim \mathcal{N}(0,\nu^2I)}&\left[\sqrt{\mu(s_i, a_i)\mu(s_j,a_j)}\phi(s_i, a_i)^{\top}\phi(s_j, a_j)\right. \\
&\left.\mathbf{1}\{\theta^{\top}\phi(s_i, a_i)\geq 0, \theta^{\top}\phi(s_j, a_j)\geq 0\}\right],
\end{aligned}
$$
where $\mathbf{1}\{\cdot\}$ is the indicator function and the expectation is taken w.r.t. the Gaussian initialization of the weights. 
According to \cite{du2018gradient/ntk}, Assumption \ref{assumption:feature-distance} ensure the positive definiteness of $G^\infty$. Let $\lambda_0$ be the minimum eigenvalue of $G^\infty$. 

Let $J^k=J_Q(\theta^k)$ be the Jacobian matrix and $J_\mu^k=U^{\frac{1}{2}}J^k$. 
We can also define the $\mu$-weighted Gram matrix in the $k$-th iteration as $G^k=J^k_\mu(J^k_\mu)^{\top}$. Let  
$
c=\frac{\min_{(s,a)}\|\sqrt{\mu(s,a)}\phi(s,a)\|}{\max_{(s,a)}\|\sqrt{\mu(s,a)}\phi(s,a)\|},
$ 
\cite{du2018gradient/ntk, zhang2019fast/ntk} suggest that $G^k\succ 0$ can be ensured by setting  the network width $m$ to be an appropriate polynomial of $|\mathcal{S}| |\mathcal{A}|$, $c$, and $\lambda_{\min}(G^\infty)^{-1}$. However, the constant $c$ is only related to $|\mathcal{S}| |\mathcal{A}|$, and does not affect the convergence rate of GNTD. Therefore, to simplify the discussion, we omit the constant $c$ in subsequent proofs. 

\begin{lemma}
\label{lemma:neural-min-eig}
(\textbf{\cite{zhang2019fast/ntk}, Lemma 6})
Suppose Assumption \ref{assumption:feature-distance} holds, then $G^\infty\succ0$. Define
$\lambda_0:=\lambda_{\min} (G^{\infty}) > 0$. If the network width $m=\Omega\left(\frac{ |\mathcal{S}| |\mathcal{A}| \log (|\mathcal{S}| |\mathcal{A}| / \delta)}{\lambda_{0}}\right)$, then we have $\lambda_{\min }(G^0) \geq$ $\frac{3}{4} \lambda_{0}$ w.p. $1-\delta$.
\end{lemma}

\begin{lemma}
\label{lemma:neural-jacobian-bound}
(\textbf{\cite{zhang2019fast/ntk}, Lemma 7}) Suppose Assumption \ref{assumption:feature-distance} holds. For any $\theta$, denote $J_\mu:=U^{\frac{1}{2}}J_Q(\theta)$, then w.p. at least $1-\delta$, we have 
$
\|J_\mu-J_\mu^0\|_{2}^{2} \leq \frac{2 |\mathcal{S}| |\mathcal{A}| B^{2 / 3}}{\nu^{2 / 3} \delta^{2 / 3} m^{1 / 3}}
$ for all $\theta$ satisfying $\|\theta-\theta^0\|_{2} \leq B$.
\end{lemma}

Lemma \ref{lemma:neural-min-eig} and Lemma \ref{lemma:neural-jacobian-bound} follow the proof of Lemma 6 and Lemma 7 in literature \cite{zhang2019fast/ntk}. Let $m=\Omega \Big(\frac{|\mathcal{S}|^3 |\mathcal{A}|^3 B^2}{\nu^2\delta^2}\Big)$, then $\|J_\mu-J_\mu^0\|_2\leq \frac{C}{3}\sqrt{\lambda_{\min} (G^{0})}$, where $C=\mathcal{O}(m^{-1/6})$ is a constant. Based on the inequality that $\sigma_{\min}(A+B)\geq \sigma_{\min}(A)-\sigma_{\max}(B)$ where $\sigma$ denotes singular value, we have
$$
\begin{aligned}
    \sigma_{\min}(J_\mu)&\geq \sigma_{\min}(J_\mu^0)-\|J_\mu-J_\mu^0\|_2 \geq \sigma_{\min}(J_\mu^0)-\frac{C}{3}\sqrt{\lambda_{\min} (G^{0})} \geq \frac{2}{3}\sqrt{\lambda_{\min} (G^{0})},
\end{aligned}
$$
where the last inequality uses the fact that $C\leq 1 $ when the network width $m$ is large enough. This means that $\lambda_{\min}(G^k)\geq \frac{1}{3}\lambda_0> 0$ if $\|\theta^k-\theta^0\|\leq B$. Thus $G^k=U^{1/2}J_Q(\theta^k)J_Q(\theta^k)^\top U^{1/2}$ is positive definite with high probability. It is slightly different from \cite{bhandari2018finite/nonasy, zou2019finite/nonasy}, which assumes that $J_Q(\theta^k)^\top U J_Q(\theta^k)$ is positive definite in the linear case.

\begin{lemma}
\label{lemma:max-q}
    Suppose Assumption \ref{assumption:feature-distance} holds. For any $k\in\{1,2,\cdots,K\}$ and $(s,a)\in\mathcal{D}_k$, we have 
    $$
    |Q(s,a;\theta^0)|^2\leq \tau_1\sqrt{\log(NK/\delta)}, \quad w.p.\ 1-\delta,
    $$
    for any $\delta\in(0,1)$.
\end{lemma}
\begin{proof}
For proof, please refer to Lemma 5.3 of \cite{cao2020generalization/ntk}. Lemma \ref{lemma:max-q} is its special case (two-layer neural network).
\end{proof}

Recalling the definition of the semi-gradient $\hat{g}(\theta, \xi)$ in Section \ref{sect:lin-proof},
we give an upper bound for it in Lemma \ref{lemma:sto-neural-bound} if $\|\theta-\theta^0\|\leq B$.

\begin{lemma}
\label{lemma:sto-neural-bound}
Conditioning on the success of the high probability events of Lemmas \ref{lemma:neural-min-eig} and \ref{lemma:neural-jacobian-bound}, where the success probability are chosen as $1-\delta_1$,
then for any $\xi$ and any $\theta$ satisfying $\|\theta-\theta^0\|_{2} \leq B$, we have w.p. $1-\delta_1$ that
$$
\|\hat{g}(\theta,\xi)\|^2\leq 9\tau_1^2(1+\gamma^2)\log(NK/\delta_1)+3(1+\gamma^2)B^2.
$$
\end{lemma}

\begin{proof}
First we compute the bounds on the gradient norm of the Q function as follows
$$
\|\nabla Q(s,a;\theta)\|_2^2=\frac{1}{m}\sum_{r=1}^m \boldsymbol{1}( \phi(s,a)^\top \theta_r>0)\|\phi(s,a)\|^2\leq 1.
$$
Consequently, by decomposing the $\theta$ in $\hat{g}(\theta;\xi)$ into $\theta^0$ and $\theta-\theta^0$ yields 
\begin{eqnarray*}
    \|\hat{g}(\theta,\xi)\|^2&= & \left\|\delta^k(\xi)\cdot\nabla Q(s,a;\theta)\right\|^2   \\
    &\overset{(i)}{\leq} & 3(\delta^0(\xi))^2+3(1+\gamma^2)\|\theta-\theta^0\|^2 \\
    &\overset{(ii)}{\leq} & 9(1+\gamma^2)\max_{(s,a)}|Q(s,a;\theta^0)|^2+9R_{\max}^2+ 3(1+\gamma^2)B^2 \\
    &\overset{(iii)}{\leq} & 9\tau_1^2(1+\gamma^2)\log(NK/\delta_1)+3(1+\gamma^2)B^2,
\end{eqnarray*}
where (i) follows $\|\nabla Q(s,a;\theta)\|^2\leq 1$ and $\|x+y+z\|^2\leq 3(\|x\|^2+\|y\|^2+\|z\|^2)$, (ii) follows the distance between $\theta$ and $\theta^0$, and (iii) follows Lemma \ref{lemma:max-q}.
\end{proof}

Let $\sigma_k=\mbox{Var}(\hat{g}(\theta^k,\xi)\mid \theta^k)\leq \max_\xi \|\hat{g}(\theta^k,\xi)\|^2$ and Lemma \ref{lemma:sto-neural-bound} implies that $\sigma_k$ is bounded if $\|\theta^k-\theta^0\|_2\leq B$. Looking back at Algorithm \ref{alg:GNTD} and \eqref{eq:population-direction}, we rewrite 
\begin{equation}
\label{eq:pop-gntd}
d_*^k=-(J_\mu^k)^\top\left(J_\mu^k(J_\mu^k)^\top\right)^{-1}\delta^k\in \arg \min_{d} L(\theta^k,d)
\end{equation}
for over-parameterized neural network approximation, where $\delta^k=Q^k-T^\pi Q^k$. The next lemma provides the gap between $d_*^k$ and $d^k$.

\begin{lemma}
\label{lemma:d*-dk}
Conditioning on the success of the high probability events of Lemma \ref{lemma:neural-min-eig} and \ref{lemma:neural-jacobian-bound}, where the success probability are chosen as $1-\delta_1$. For any $\delta_2>0$, we choose $\omega\in(0,2)$ and the sample size $|\mathcal{D}_k|=N\geq\frac{24}{\omega^2}\log \frac{2md}{\delta_2}$ for $k$-th iteration. If the iteration $\|\theta^k-\theta^0\|_{2} \leq B$ and the network width $m=\Omega \Big(\frac{|\mathcal{S}|^3 |\mathcal{A}|^3 B^2}{\nu^2\delta_1^2}\Big)$, we have $w.p.\ 1-(\delta_1+2\delta_2)$ that
$$
\|d^k_*-d^k\|\leq \frac{9\tau_2\sqrt{3\log(NK/\delta_1)\log(1/\delta_2)}}{4\lambda_{\min}(G^0)\sqrt{|\mathcal{D}_k|}}+  \frac{243\omega \tau_2\sqrt{\log(NK/\delta_1)}}{16\lambda_{\min}(G^0)^2},
$$
where $\tau_2^2=9\tau_1^2(1+\gamma^2)+3(1+\gamma^2)B^2$.
\end{lemma}

\begin{proof}
Recall that for neural network approximation, $d^k=(\hat{H}^k+\omega I)^{-1}\hat{g}^k$ with $\hat{H}^k=(J^k)^\top \hat{U}^kJ^k$. Observing that $J_\mu^k d^k_*=-\delta^k$, we can compute
\begin{eqnarray}
\label{eq:lin-over-1-step}
    &\ &J_\mu^k (d^k_*-d^k)=J^k_\mu (\hat{H}^k+\omega I)^{-1}\hat{g}^k-(Q^k-T^\pi Q^k) \\
    &=& J^k_\mu (\hat{H}^k+\omega I)^{-1}\left(\hat{g}^k-g(\theta^k)\right)+\left(J^k_\mu (\hat{H}^k+\omega I)^{-1}(J^k_\mu)^\top-I\right)U^{\frac{1}{2}}\delta^k. \nonumber
\end{eqnarray}
By Lemma \ref{lemma:matrix-bernstein} when the sample size $|\mathcal{D}_k|\geq \frac{24}{\eta^2}\log \frac{2d}{\delta_2}$ for any $\eta\in(0,1)$, we have $w.p.\ 1-\delta_2$ that $\|(J^k)^\top\hat{U}^kJ^k-(J^k)^\top UJ^k\| \leq \eta$. Let $\Delta_k=(J^k)^\top\hat{U}^kJ^k+\omega I-(J^k)^\top UJ^k$ and $\omega=2\eta$, and we know that $\Delta_k$ is invertible with high probability. Then by the Sherman-Morrison-Woodbury (SMW) formula, we have
\begin{eqnarray*}
J^k_\mu(\hat{H}^k+\omega I)^{-1}&=&J^k_\mu\left(\Delta_k^{-1}-\Delta_k^{-1}(J^k_\mu)^\top(I+J^k_\mu\Delta_k^{-1}(J^k_\mu)^\top)^{-1}J^k_\mu\Delta_k^{-1}\right) \\
&=&(I+J^k_\mu\Delta_k^{-1}(J^k_\mu)^\top)^{-1}J^k_\mu \Delta_k^{-1}.
\end{eqnarray*}
Consider the singular value decomposition, we write $J^k_\mu=U_1\Sigma_1V_1^\top, \Delta_k=U_2\Sigma_2U_2^\top$. Then, 
$$
\begin{aligned}
\lambda_{\min}(J^k_\mu\Delta_k^{-1}(J^k_\mu)^\top)\ =&\ \lambda_{\min}(\Sigma_1V_1^\top U_2\Sigma_2^{-1}U_2^\top  V_1\Sigma_1^\top) \\ 
    \geq&\  \lambda_{\min}(\Sigma_1\Sigma_1^\top)\lambda_{\min}(V_1^\top U_2\Sigma_2^{-1}U_2^\top  V_1) 
    \ \geq\  \frac{\lambda_{\min} (G^k)}{3\eta},    
\end{aligned}
$$
where the last inequality is because $\|\Delta_k\|=\|(J^k)^\top\hat{U}^kJ^k+\omega I-(J^k)^\top UJ^k\|\leq \omega+\|(J^k)^\top\hat{U}^kJ^k-(J^k)^\top UJ^k\|\leq \omega+\eta=3\eta$.
Thus by Lemma \ref{lemma:vector-bernstein}, we have $w.p.\ 1-2\delta_2$ that
\begin{equation}
\label{eq:lin-over-semi-grad}
\begin{aligned}
\left\| J^k_\mu(\hat{H}^k+\omega I)^{-1}\right.&\left.\left(\hat{g}^k-g(\theta^k)\right)\right\|
\leq   \|(I+J^k_\mu\Delta_k^{-1}(J^k_\mu)^\top)^{-1}J^k_\mu \Delta_k^{-1}\|\cdot \left\|\left(\hat{g}^k-g(\theta^k)\right)\right\| \\
\leq &\  \frac{\|(I+J^k_\mu\Delta_k^{-1}(J^k_\mu)^\top)^{-1}J^k_\mu \Delta_k^{-1}(J^k_\mu)^\top\|}{\sqrt{\lambda_{\min} (G^k)}}\cdot \left\|\left(\hat{g}^k-g(\theta^k)\right)\right\| \\
\leq &\  \frac{\sigma_k\sqrt{3\log(1/\delta_2)}}{\sqrt{\lambda_{\min} (G^k)|\mathcal{D}_k|}},
\end{aligned}    
\end{equation}
and 
\begin{eqnarray}
\label{eq:lin-over-mat}
    \left\|\left(J^k_\mu (\hat{H}^k+\omega I)^{-1}J_\mu^\top-I\right)U^{\frac{1}{2}}\delta^k\right\| &\leq &\left\|(I+J^k_\mu\Delta_k^{-1}(J^k_\mu)^\top)^{-1}\right\|\cdot\left\|U^{\frac{1}{2}}\delta^k\right\| \nonumber\\
    &\leq& \frac{3\eta \|g(\theta_k)\|}{\lambda_{\min} (G^k)^{3/2}}  \leq  \frac{3\eta \sigma_k}{\lambda_{\min} (G^k)^{3/2}},
\end{eqnarray}
where the second inequality uses the fact that $g(\theta^k)=(J^k_\mu)^\top U^{\frac{1}{2}}\delta^k$.
Plugging \eqref{eq:lin-over-semi-grad}, \eqref{eq:lin-over-mat} into \eqref{eq:lin-over-1-step} yields that given $\theta^k$,
\begin{eqnarray*}
&\ &\|d^k_*-d^k\| \leq \frac{\|J_\mu^k(d^k_*-d^k)\|}{\sqrt{\lambda_{\min} (G^k)}} \\
&\leq & \frac{1}{\sqrt{\lambda_{\min} (G^k)}}\left(\left\|J^k_\mu (\hat{H}^k+\omega I)^{-1}\left(\hat{g}^k-g(\theta^k)\right)\right\|+\left\|\left(J^k_\mu (\hat{H}^k+\omega I)^{-1}(J^k_\mu)^\top-I\right)U^{\frac{1}{2}}\delta^k\right\|\right) \\
&\overset{(i)}{\leq} & \frac{\sigma_k\sqrt{3\log(1/\delta_2)}}{\lambda_{\min} (G^k)\sqrt{|\mathcal{D}_k|}}+  \frac{3\eta \sigma_k}{\lambda_{\min} (G^k)^{2}} \ \overset{(ii)}{\leq} \  \frac{9\tau_2\sqrt{3\log(NK/\delta_1)\log(1/\delta_2)}}{4\lambda_{\min}(G^0)\sqrt{|\mathcal{D}_k|}}+  \frac{243\omega \tau_2\sqrt{\log(NK/\delta_1)}}{16\lambda_{\min}(G^0)^2}
\end{eqnarray*}
$w.p.\ 1-(\delta_1+2\delta_2)$, where (i) follows \eqref{eq:lin-over-semi-grad} and \eqref{eq:lin-over-mat}, and (ii) follows $\lambda_{\min} (G^k)\geq \frac{4}{9}\lambda_{\min} (G^0)$. This completes the proof.
\end{proof} 


\begin{lemma}
\label{lemma:pop-neural-network-one-step}
Recall the definition of $Q^k_*$ in Section \ref{sect:lin-proof}.
Conditioning on the success of the high probability events of Lemma \ref{lemma:neural-min-eig} and \ref{lemma:neural-jacobian-bound}, where the success probability are chosen as $1-\delta_1$, if $\|\theta^k-\theta^0\|_{2} \leq B$ and the network width $m=\Omega \Big(\frac{|\mathcal{S}|^3 |\mathcal{A}|^3 B^2}{\nu^2\delta_1^2}\Big)$, then the following inequalities hold
\begin{equation}
    \label{eq:lemma:pop-neural-network-one-step-1}
    \left\|Q^{k+1}_*-Q^k-\beta\left(T^\pi Q^k-Q^k\right)\right\|_\mu\leq \beta C\|Q^k-T^\pi Q^k\|_\mu,
\end{equation}
where $C=\mathcal{O}(m^{-1/6})$ is the constant defined in the preceding text.
\end{lemma}

\begin{proof}
Let $\theta(s):=s\theta^{k+1}_*+(1-s)\theta^k$ and calculate
$$
\begin{aligned}
&U^{1/2}(Q^{k+1}_*-Q^k)=-\beta U^{1/2}\int_{s=0}^{1}  J_Q(\theta(s)) (J^k_\mu)^{\top}   \left(J^k_\mu(J^k_\mu)^{\top}\right)^{-1} U^{\frac{1}{2}}(Q^k-T^\pi Q^k)   d s\\
=&\underbrace{-\beta U^{1/2}\int_{s=0}^{1} J_Q(\theta^k)  (J^k_\mu)^{\top} \left(J^k_\mu(J^k_\mu)^{\top}\right)^{-1} U^{\frac{1}{2}}(Q^k-T^\pi Q^k) d s}_{-\beta U^{1/2}(Q^k-T^\pi Q^k)}\\
&\underbrace{-\beta U^{1/2}\int_{s=0}^{1}\left( J_Q(\theta(s)) - J_Q(\theta^k)\right) (J^k_\mu)^{\top} \left(J^k_\mu(J^k_\mu)^{\top}\right)^{-1} U^{1/2}(Q^k-T^\pi Q^k) d s}_{\textcircled{1}}.
\end{aligned}
$$

Next we estimate the bound on the norm of the second term $\textcircled{1}$. Note that
\begin{eqnarray*}
    \|\textcircled{1}\|_2&\leq & \beta \left\| \int_{s=0}^{1} \left(J_\mu(\theta(s))-J^k_\mu\right) d s \right\|_{2}\cdot \left\|(J^k_\mu)^{\top} \left(J^k_\mu(J^k_\mu)^{\top}\right)^{-1}\right\|_2 \cdot \|Q^k-T^\pi Q^k\|_\mu \\
    &\leq & \beta \left(\left\| \int_{s=0}^{1} \left(J_\mu(\theta(s))-J^0_\mu\right) d s \right\|_{2}+\left\| \int_{s=0}^{1} \left(J^0_\mu-J^k_\mu\right) d s \right\|_{2}\right)\cdot \\
    &\ &\left\|(J^k_\mu)^{\top} \left(J^k_\mu(J^k_\mu)^{\top}\right)^{-1}\right\|_2 \cdot \|Q^k-T^\pi Q^k\|_\mu \\
    &\overset{(i)}{\leq} & \beta \cdot \frac{2}{3}C\cdot \sqrt{\lambda_{\min}(G^0)} \cdot \frac{1}{\sqrt{\lambda_{\min}(G^k)}} \|Q^k-T^\pi Q^k\|_\mu \overset{(ii)}{\leq} \beta C \|Q^k-T^\pi Q^k\|_\mu,
\end{eqnarray*}
where (i) is due to $\|J_\mu-J_\mu^0\|_2\leq \frac{C}{3}\sqrt{\lambda_{\min}(G^0)}$ and (ii) is due to $\lambda_{\min}(G)\geq \frac{4}{9}\lambda_{\min} (G^0)$.
Thus we complete the proof. 
\end{proof}

Now we are ready to provide the proof of the Theorem \ref{theorem:sto-neural-netwrk}. 

\begin{proof}
Note that the key results of Lemmas \ref{lemma:sto-neural-bound}$\sim$\ref{lemma:pop-neural-network-one-step} all rely on the condition that the analyzed point $\theta$ stays close to $\theta^0$. Therefore, to prove this theorem with the above lemmas, we let $B=\frac{24\|Q^*-T^\pi Q^0\|_\mu}{(1-\gamma)\sqrt{\lambda_{\min}(G^0)}}$ and will need to prove the following argument by induction:
\begin{equation}
    \label{eq:theorem:pop-neural-network-1}
    \|\theta^k-\theta^0\|_2\leq \frac{24\|Q^*-T^\pi Q^0\|_\mu}{(1-\gamma)\sqrt{\lambda_{\min}(G^0)}},\qquad \forall k\geq 0.
\end{equation} 
Then the final convergence rate result will be automatically covered as a byproduct in the proof of \eqref{eq:theorem:pop-neural-network-1}. When $k=0$, \eqref{eq:theorem:pop-neural-network-1} is obviously true. Now, suppose \eqref{eq:theorem:pop-neural-network-1} holds for all $k=0,1,\cdots,t$, we prove this argument for $k=t+1$.

Define the operator $T_\beta^\pi Q:= (1-\beta)Q+\beta T^\pi Q$. Conditioning on the success of the high probability event in Lemmas \ref{lemma:neural-min-eig} and \ref{lemma:neural-jacobian-bound}, then Lemma \ref{lemma:pop-neural-network-one-step} and \eqref{eq:theorem:pop-neural-network-1} indicates that \eqref{eq:lemma:pop-neural-network-one-step-1} holds for $k=0,1,...,t$. Next we set $\delta_1=\frac{\delta}{3}$ and $\delta_2=\frac{\delta}{3K}$. Then for any $0\leq k\leq t$, there exists a constant $C_2=\mathcal{O}((1-\gamma)B)$ such that
\begin{equation}
\label{eq:pop-neural-Q-Qpi}
\begin{aligned}
&\|Q^{k+1}-T_\beta^\pi Q^k\|_\mu \\
\leq&\  \|Q^{k+1}-Q^{k+1}_*\|_\mu + \|Q^{k+1}_*-Q^k-\beta(Q^k-T^\pi Q^k)\|_\mu \\
\ \overset{(i)}{\leq}& \ \beta \|d^{k}-d^k_*\|+ \beta C\|Q^k-T^\pi Q^k\|_\mu\\
\ \overset{(ii)}{\leq}& \  \frac{9\beta\tau_2\sqrt{3\log(NK/\delta_1)\log(1/\delta_2)}}{4\lambda_{\min}(G^0)\sqrt{|\mathcal{D}_k|}}+  \frac{243\beta\omega \tau_2\sqrt{\log(NK/\delta_1)}}{16\lambda_{\min}(G^0)^2} \\
\ &\  +\beta C\left(\|Q^0-TQ^0\|_\mu+2B\right) \\
\ \overset{(iii)}{\leq}& \ \frac{\beta C_2}{1-\gamma} \left(\frac{\log(NK/\delta)}{\sqrt{|\mathcal{D}_k|}}+\omega+m^{-\frac{1}{6}}\right)
\end{aligned}
\end{equation}
$w.p.\ 1-\delta_1-2t\delta_2$, where (i) follows $\|\nabla Q(s,a;\theta)\|\leq 1$ and Lemma \ref{lemma:pop-neural-network-one-step}, (ii) follows Lemma \ref{lemma:d*-dk} and $\|\theta^k-\theta^0\|\leq B$, and (iii) follows $C=\mathcal{O}(m^{-1/6})$. For the accuracy level $\varepsilon$, we choose $|D_k|=N=\Omega(\varepsilon^{-1}), \omega=\Omega (\varepsilon^{0.5}),m=\Omega(\varepsilon^{-3})$. According to the inductive assumption, $\|\theta^k-\theta^0\|\leq B$, then 
\begin{eqnarray*}
\|Q^{k+1}-Q^\pi\|_\mu &\leq& \|T_\beta^\pi Q^k-Q^\pi\|_\mu + \|T_\beta^\pi Q^k-Q^{k+1}\|_\mu \\
&\leq&(1-(1-\gamma)\beta)\|Q^k-Q^\pi\|_\mu+\frac{\beta C_2}{1-\gamma} \left(\frac{\log(NK/\delta)}{\sqrt{|\mathcal{D}_k|}}+\omega+m^{-\frac{1}{6}}\right).
\end{eqnarray*}
Therefore, for $0\leq k\leq t$, we have
\begin{equation}
\label{eq:Q-Qpi}
\|Q^{k+1}-Q^\pi\|_\mu \leq (1-(1-\gamma)\beta)^{k+1}\|Q^0-Q^\pi\|_\mu+\frac{C_2}{(1-\gamma)^2} \left(\frac{\log(NK/\delta)}{\sqrt{N}}+\omega+m^{-\frac{1}{6}}\right).
\end{equation}
In the same way, we can estimate the Bellman error of $\theta^{k+1}$,
\begin{eqnarray}
&\ & \|Q^{k+1}-T^\pi Q^{k+1}\|_\mu \leq \|Q^{k+1}-T_\beta^\pi Q^k\|_\mu+\|T^\pi Q^{k+1}-T^\pi Q^k\|_\mu+(1-\beta)\|Q^k-T^\pi Q^k\|_\mu \nonumber\\
&\leq & \|Q^{k+1}-T_\beta^\pi Q^k\|_\mu+\gamma\left(\|Q_*^{k+1}- Q^k\|_\mu+\|Q_*^{k+1}- Q^{k+1}\|_\mu\right)+(1-\beta)\|Q^k-T^\pi Q^k\|_\mu \nonumber\\
&\overset{(i)}{\leq} & \frac{2\beta C_2}{1-\gamma} \left(\frac{\log(NK/\delta)}{\sqrt{N}}+\omega+m^{-\frac{1}{6}}\right)+\gamma \beta (1+C) \|Q^k-T^\pi Q^k\|_\mu+(1-\beta)\|Q^k-T^\pi Q^k\|_\mu \nonumber\\
&\leq &\frac{2\beta C_2}{1-\gamma} \left(\frac{\log(NK/\delta)}{\sqrt{N}}+\omega+m^{-\frac{1}{6}}\right)+(1-(1-\gamma)\beta/2)\|Q^k-T^\pi Q^k\|_\mu\nonumber \\
&\leq & \frac{4 C_2}{(1-\gamma)^2} \left(\frac{\log(NK/\delta)}{\sqrt{N}}+\omega+m^{-\frac{1}{6}}\right)+(1-(1-\gamma)\beta/2)^{k+1}\|Q^0-T^\pi Q^0\|_\mu \nonumber
\end{eqnarray}

Let $\varepsilon$ is small enough and the positive integer $K=\mathcal{O}(\log \frac{1}{\varepsilon})\ll \varepsilon^{-0.5}$ such that 
\begin{equation}
\label{eq:acc-level}
\frac{12C_2 K}{(1-\gamma)^2\sqrt{\lambda_{\min}(G^0)}}\left(\frac{\log(NK/\delta)}{\sqrt{N}}+\omega+m^{-\frac{1}{6}}\right) \leq \frac{12\|Q^*-T^\pi Q^0\|_\mu}{(1-\gamma)\sqrt{\lambda_{\min}(G^0)}}.
\end{equation}
consequently, for $k=t+1$, we have $w.p.\ 1-(\delta_1+2K\delta_2)=1-\delta$ that
\begin{eqnarray*}
\|\theta^{t+1}-\theta^0\|_2&\leq& \beta\sum_{s=0}^{t}\|d^s_*\|_2+\|d^s_*-d^s\|_2 \\
&\leq & \beta \sum_{s=0}^{t} \frac{\|J_\mu^s d^s_*\|_\mu}{\sqrt{\lambda_{\min}(G^s)}}+ \|d^s_*-d^s\|_2\\
&\overset{(i)}{\leq} & \beta \sum_{s=0}^{t}\left( \frac{\|Q^s-T^\pi Q^s\|_\mu}{\sqrt{\frac{4}{9}\lambda_{\min}(G^0)}}+ \frac{C_2}{1-\gamma} \left(\frac{\log(NK/\delta)}{\sqrt{N}}+\omega+m^{-\frac{1}{6}}\right)\right) \\ 
&\overset{(ii)}{\leq} & \beta \sum_{s=0}^{t}\left( \frac{(1-\frac{(1-\gamma)\beta}{2})^s\|Q^0-T^\pi Q^0\|_\mu}{\sqrt{\frac{4}{9}\lambda_{\min}(G^0)}} + \frac{4 C_2}{(1-\gamma)^2\sqrt{\frac{4}{9}\lambda_{\min}(G^0)}}\cdot \right.\\
&\ &\left.\left(\frac{\log(NK/\delta)}{\sqrt{N}}+\omega+m^{-\frac{1}{6}}\right)+\frac{C_2}{1-\gamma}\left(\frac{\log(NK/\delta)}{\sqrt{N}}+\omega+m^{-\frac{1}{6}}\right)\right) \\ 
&= & \beta \sum_{s=0}^{t} \frac{6(1-\frac{(1-\gamma)\beta}{2})^{s}\|Q^0-T^\pi Q^0\|_\mu}{\sqrt{\lambda_{\min}(G^0)}}+\left(\frac{6}{(1-\gamma)\sqrt{\lambda_{\min}(G^0)}}+1\right)\cdot \\
&\ & \frac{\beta C_2(t+1)}{1-\gamma}\left(\frac{\log(NK/\delta)}{\sqrt{N}}+\omega+m^{-\frac{1}{6}}\right) \\
&\overset{(iii)}{\leq} & \frac{24\|Q^*-T^\pi Q^0\|_\mu}{(1-\gamma)\sqrt{\lambda_{\min}(G^0)}}, 
\end{eqnarray*}
where (i) follows $J_\mu^s d_*^s=Q^s-T^\pi Q^s$ and Lemma \ref{lemma:d*-dk} (or \eqref{eq:pop-neural-Q-Qpi}), (ii) follows \eqref{eq:Q-Qpi}, and (iii) follows \eqref{eq:acc-level}. Therefore, the statement \eqref{eq:theorem:pop-neural-network-1} holds.

The above derivation verifies that when the conditions in Theorem \ref{theorem:sto-neural-netwrk} hold, $\theta^k$ will always stay close to the initialization parameters $\theta^0$ with high probability. Thus the above lemmas in this subsection and the inequality \eqref{eq:Q-Qpi} are all correct for any $0\leq k\leq K-1$. Thus we complete the proof of \ref{theorem:sto-neural-netwrk}. 
\end{proof}

\subsection{Proof of Theorem \ref{theorem:sto-smooth}}
In this subsection, we will continue to use the notation of the previous section. Looking back at the definition of $H(\theta)$ and $g(\theta)$, we rewrite the unique closed form solution $d^k_*=-\left[H(\theta^k)\right]^{-1}g(\theta^k)=-\left((J^k)^\top U J^k\right)^{-1}$ $(J^k)^\top U \delta^k=\arg\min_d L(\theta^k, d)$ under Assumption \ref{assumption:parameter-curve-min-eig}. The following lemma provides the gap between $d^k_*$ and $d^k$.

\begin{lemma}
\label{lemma:smooth-d*-dk}
Under Assumptions \ref{assumption:smooth} and \ref{assumption:parameter-curve-min-eig}, if we choose the damping term $\omega\in(0,1)$, the sample size $|\mathcal{D}_k|=N=\Omega\left(\frac{6L_1^4}{\omega^2}\log \frac{2d}{\delta_1}\right)$ for any $\delta_1\in(0,1)$, then given $\theta^k$, we have that
$$
\|d^k_*-d^k\|\leq \frac{2M_1\omega}{\lambda_0^2}+\frac{M_1}{\lambda_0}\cdot \sqrt{\frac{3\log (1/\delta_1)}{N}}, \quad w.p.\ 1-2\delta_1,
$$
where $M_1:=L_1 \left((1+\gamma)M+R_{\max}\right)$.
\end{lemma}

\begin{proof}
According to the stochastic update of GNTD, we have 
\begin{equation}
\label{eq:smooth-sto-gap}
\begin{aligned}
\|d^k_*-d^k\| \ \leq&\  \|d^k_*+(\hat{H}^k+\omega I)^{-1}g(\theta^k)\|+\|d^k+(\hat{H}^k+\omega I)^{-1}g(\theta^k)\| \\
\ \leq&\ \left\|\left(H(\theta^k)\right)^{-1}-(\hat{H}^k+\omega I)^{-1}\right\|\cdot \|g(\theta^k)\| + \|(\hat{H}^k+\omega I)^{-1}\|\cdot \|\hat{g}^k-g(\theta^k)\| \\
\ \overset{(i)}{\leq} &\ M_1 \left\|\left(H(\theta^k)\right)^{-1}-(\hat{H}^k+\omega I)^{-1}\right\| + \|(\hat{H}^k+\omega I)^{-1}\|\cdot \|\hat{g}^k-g(\theta^k)\|,
\end{aligned}
\end{equation}
where (i) follows Assumption \ref{assumption:smooth} and
$$
\|g(\theta^k)\|=\|(J^k)^\top U(Q^k-T^\pi Q^k)\|\leq \|J^k_\mu\|\cdot \|Q^k-T^\pi Q^k\|_\mu \leq L_1 \left((1+\gamma)M+R_{\max}\right),
$$
and $M_1:=L_1 \left((1+\gamma)M+R_{\max}\right)$.
By Assumption \ref{assumption:parameter-curve-min-eig} and Lemma \ref{lemma:matrix-bernstein}, if the sample size $|\mathcal{D}_k|=N\geq\frac{6L_1^4}{\eta^2}\log \frac{2d}{\delta_1}$ for any $\eta\in(0,1)$, we have
\begin{equation}
\label{eq:smooth-gn-matrix-approx}
\begin{aligned}
\lambda_{\min}((J^k)^{\top}\hat{U}^k J^k)\ \geq&\  \lambda_{\min}((J^k)^{\top}U J^k)-\|(J^k)^{\top}\hat{U}^k J^k-(J^k)^{\top}U J^k\|_2\\
\ \geq&\  \lambda_0-\eta,\quad w.p.\ 1-\delta_1.
\end{aligned}
\end{equation}
Let $\eta=\omega$. Then we have $w.p.\ 1-\delta_1$ that
\begin{equation}
\label{eq:smooth-matrix}
\begin{aligned}
\ &\left\|\left((\hat{H}^k+\omega I)^{-1}-\left((J^k)^{\top} U J^k\right)^{-1}\right)\right\|_2 \\
=&\ \frac{1}{\lambda_0^{1/2}}\left\|U^{1/2}J^k\left((\hat{H}^k+\omega I)^{-1}-\left((J^k)^{\top} U J^k\right)^{-1}\right)\right\|_2 \\
=&\ \frac{1}{\lambda_0^{1/2}}\Big\|U^{1/2}J^k\big((J^k)^{\top} U (J^k)\big)^{-1}\big((J^k)^{\top}\hat{U}^kJ^k+\omega I\\
&-(J^k)^{\top} U J^k\big)\big((J^k)^{\top}\hat{U}^kJ^k+\omega I\big)^{-1}\Big\|_2  \\
\leq &\ \frac{1}{\lambda_0^{1/2}}\left\|U^{1/2}J^k\left((J^k)^{\top} U J^k\right)^{-1}\right\|_2\cdot \left\|(J^k)^{\top}\hat{U}^kJ^k+\omega I\right. \\
&\left.-(J^k)^{\top} U J^k\right\|_2\cdot\Big\|\left((J^k)^{\top}\hat{U}^kJ^k+\omega I\right)^{-1}\Big\|_2  \\
\overset{(i)}{\leq} &\  \frac{1}{\lambda_0}\cdot(\eta+\omega)\cdot \frac{1}{\lambda_0+\omega-\eta} 
= \frac{1}{\lambda_0^{2}}(\eta+\omega), 
\end{aligned}
\end{equation}
where (i) is due to \eqref{eq:smooth-gn-matrix-approx}.
Next by Lemma \ref{lemma:vector-bernstein}, we have
\begin{equation}
\label{eq:smooth-vector}
\|g(\theta^k)-\hat{g}^k\|\leq \sigma_k\sqrt{\frac{3\log (1/\delta_1)}{N}}\overset{(i)}{\leq } M_1 \sqrt{\frac{3\log (1/\delta_1)}{N}},\qquad w.p.\ 1-\delta_1,
\end{equation}
where (i) is due to $\sigma_k^2=\mbox{Var}(\hat{g}(\theta^k,\xi)\mid \theta^k)\leq \max_\xi \|g(\theta^k,\xi)\|^2\leq M_1^2$.
Thus plugging \eqref{eq:smooth-matrix}, \eqref{eq:smooth-vector} into \eqref{eq:smooth-sto-gap} yields that 
\begin{eqnarray*}
\|d^k_*-d^k\|&\leq & \frac{2M_1\omega}{\lambda_0^2}+\frac{M_1}{\lambda_0}\cdot \sqrt{\frac{3\log (1/\delta_1)}{N}}, \quad w.p.\ 1-2\delta_1.
\end{eqnarray*} 
\end{proof}

Similar to the previous section, we give the following lemma for obtaining upper bounds on $\|Q^{k+1}_*-T^\pi_\beta Q^k\|_\mu$. 

\begin{lemma}
\label{lemma:smooth-err2-err3}
Recall that $\theta^{k+1}=\theta^k+\beta d^k_*$ and the Q-function $Q^{k+1}_*(s,a)=Q(s,a;$ $\theta^{k+1}_*)$. Under Assumptions \ref{assumption:smooth} and \ref{assumption:parameter-curve-min-eig}, we have that for any given $\theta^k$,
$$
\|Q^{k+1}_*-Q^k-\beta (T^\pi Q^k-Q^k)\|_\mu \leq \frac{\beta^2L_2M_1^2}{\lambda_0^2} + \beta \varepsilon_{\mathcal{F}}.
$$
\end{lemma}

\begin{proof}
Recall the population TD error $\delta^k=Q^k-T^\pi Q^k$ and 
$$d^k = - \left((J^k)^\top UJ^k\right)^{-1}(J^k)^{\top}U\delta^k,$$
under Assumptions \ref{assumption:smooth} and \ref{assumption:parameter-curve-min-eig}. Consequently,
\begin{eqnarray*}
    \|d^k_*\|&\leq & \lambda_{\min}\left((J^k)^\top UJ^k\right)^{-1}\cdot \left\|(J^k)^{\top}U\delta^k\right\| \\
    &\leq & \frac{1}{\lambda_0}\left\|J^k\right\|_\mu\cdot \left\|Q^k-T^\pi Q^k\right\|_\mu \\
    &\leq& \frac{L_1}{\lambda_0}\left((1+\gamma)M+R_{\max}\right).
\end{eqnarray*}
Define the residual term $f_k(s,a)\!:=\!Q(s,a;\theta^{k+1}_*)\!-\!Q(s,a;\theta^k)\!-\!\left\langle\nabla Q(s,a;\theta^k), \theta^{k+1}_*\!-\!\theta^k\right\rangle$. By Assumptions \ref{assumption:smooth} and \ref{assumption:parameter-curve-min-eig}, we have
$$
\|f_k(s,a)\|_\mu \leq L_2\|\theta^{k+1}_*-\theta^k\|^2=L_2\beta^2\|d^k_*\|^2\leq \frac{\beta^2L_2M_1^2}{\lambda_0^2} .
$$
Recall the notation $T_\beta^\pi Q=(1-\beta)Q+\beta T^\pi Q$, with $\varepsilon_{\mathcal{F}}$ defined by \eqref{eq:fitted-error} in Section \ref{sect:smooth-conv}, we have 
\begin{eqnarray*}
    \|Q^{k+1}_*-T_\beta^\pi Q^k\|_\mu&=&\left\|Q^{k+1} - Q^k - \beta J^kd^k_* + \beta(J^kd^k_* + Q^k-T^\pi Q^k)\right\|_\mu\\
    &\leq&  \|f_k\|_\mu + \beta L(\theta^k,d^k)^{1/2} \\
    &\leq& \frac{\beta^2L_2M_1^2}{\lambda_0^2} + \beta \varepsilon_{\mathcal{F}}.
\end{eqnarray*}
Thus we complete the proof.
\end{proof}

Now we are ready to provide the proof of the Theorem \ref{theorem:sto-smooth}.
\begin{proof}
Let $\delta_1=\frac{\delta}{2K}$ and we have $w.p.\ 1-2K\delta_1=1-\delta$ that for any $0\leq k\leq K-1$, 
\begin{eqnarray*}
\|Q^{k+1}-Q^\pi\|_\mu &\leq& \|T_\beta^\pi Q^k-Q^\pi\|_\mu + \|Q^{k+1}-T_\beta^\pi Q^k\|_\mu \\
&\leq& (1-(1-\gamma)\beta)\|Q^k-Q^\pi\|_\mu + \|Q^{k+1}-Q^{k+1}_*\|_\mu + \|Q^{k+1}_*-T_\beta^\pi Q^k\|_\mu \\
&\overset{(i)}{\leq}& (1-(1-\gamma)\beta)\|Q^k-Q^\pi\|_\mu+ L_1\left(\frac{2M_1\omega}{\lambda_0^2}+\frac{M_1}{\lambda_0}\cdot \sqrt{\frac{3\log (2K/\delta)}{N}}\right)\\
&\ &+\frac{\beta^2L_2M_1^2}{\lambda_0^2} + \beta \varepsilon_{\mathcal{F}} \\
&\overset{(ii)}{\leq}& (1-(1-\gamma)\beta)\|Q^k-Q^\pi\|_\mu+ \beta C_3 \left(\omega+\sqrt{\frac{\log (2K/\delta)}{N}}+\varepsilon_{\mathcal{F}}+\beta\right),
\end{eqnarray*}
where (i) is due to Lemmas \ref{lemma:smooth-d*-dk} and \ref{lemma:smooth-err2-err3}, and (ii) is due to a given constant $C_3=\mathcal{O}(1)$ that satisfies the above inequality. Thus we complete the proof of Theorem \ref{theorem:sto-smooth}.
\end{proof}

\begin{remark}
Note that $|Q(s,a;\theta)|\leq M$ in Assumption \ref{assumption:smooth} is not necessary. Similar to the proof in neural network approximation, it is natural to prove $\|\theta^k-\theta^0\|\leq B, \forall k$ by induction, where $B>0$ is a given constant. Then we have that $|Q(s,a;\theta^k)|\leq |Q(s,a;\theta^k)-Q(s,a;\theta^0)|+|Q(s,a;\theta^0)|\leq L_1 B + |Q(s,a;\theta^0)|$ is bounded for any $(s,a),\theta^k$. Therefore, we add this assumption to omit redundant discussion for general smooth function approximation.
\end{remark}

\section{An Efficient Implementation of Neural GNTD}
\label{sect:kfac}
In this section, we propose a new algorithm that utilizes the Kronecker-Factored Approximate Curvature (K-FAC) method \cite{martens2015optimizing/kfac} to provide an efficient implementation of neural GNTD.

In Section \ref{section:GNTD-method}, we introduce the generalized GNTD algorithm, which encounters the issue of high complexity due to matrix inversion. However, we observe that GNTD has a similar structure to the natural gradient method, despite from the natural gradient method in terms of expression (natural gradient method requires the assumption that the loss function is the negative log-likelihood of normal distribution). This observation inspires us to refine the neural GNTD algorithm.

\begin{algorithm}[htbp] 
\caption{Forward and backward pass of $Q$ function for a single state-action pair $(s, a)$} 
\label{alg:forward-backward-nn} 
\begin{algorithmic}[1] 
\STATE \textbf{Input:} $(s,a)$, weights (and biases) $\theta_l$ and ReLU function $\psi$.
\STATE $p_0=\phi(s,a)$, where $\phi$ is a feature mapping. 
\FOR{$l=1,\cdots,L$}
\STATE $\overline{p}_{l-1}=\left(p_{l-1}, 1\right);h_{l}=\textbf{\rm Mat}(\theta_{l}) \overline{p}_{l-1} ; p_{l}=\psi\left(h_{l}\right)$.
\ENDFOR
\STATE $\mathcal{D} p_L=e$, where $e$ is an all-ones vector.
\FOR{$l=L,\cdots,1$}
\STATE $q_{l}=\mathcal{D} p_{l} \odot \psi^{\prime}\left(h_{l}\right) ; \mathcal{D} \theta_{l}=\textbf{\rm Vec}(q_{l} \overline{p}_{l-1}^{\top}) ; \mathcal{D} p_{l-1}=\textbf{\rm Mat}(\theta_{l})^{\top} q_{l}$.
\ENDFOR
\end{algorithmic}
\end{algorithm}

For a feed-forward deep neural network (DNN) with $L$ layers, we denote the weight matrices as $\theta_{l}$ of $l$-th layer ($l=1,2,\cdots,L$) and we denote the ReLU activation function as $\psi(\cdot)$. For any state-action pair $(s, a)$, the output $Q(s,a;\theta)$ is in general a non-convex function of the weights $\theta=\big[\theta_{1}^{\top}, \ldots, \theta_{L}^{\top}\big]^{\top}$. Alternatively, $\theta_l$ can also be viewed as an $\mathbb{R}^{(n_{l-1}+1)\times n_l}$ parameter matrix that maps $n_{l-1}$-dimensional vectors to ${n_{l}}$-dimensional vectors. We define $\textbf{\rm Mat}(\cdot)$ as a matrix form of the vector parameters related to the number of neurons in a single layer and define $\textbf{\rm Vec}(\cdot)$ as a flattened vector form of the matrix parameters. Let $\odot$ and $\otimes$ be the Hadamard product and the Kronecker product respectively. The following algorithm describes the forward and backward pass of Q-value for a single state-action pair $(s, a)$.

From Algorithm \ref{alg:forward-backward-nn}, with the weights of the neural network being $\theta^k$, we let $p^k_l$ and $q^k_l$ denote the forward vector and backward vector of the $l$-th layer, respectively, and we define the matrices
\begin{equation}
    \label{eq:PQ}
    P^k_l:=p^k_l(p^k_l)^{\top} \qquad\mbox{and}\qquad Q^k_l:=q^k_l(q^k_l)^{\top}.
\end{equation}
For a training dataset that contains multiple data-points, the K-FAC method \cite{martens2015optimizing/kfac} attempts to approximate the matrix $H(\theta^k)$ in \eqref{eq:H-and-g} by the following block-diagonal matrix 
$$
H(\theta^k)\approx \textbf{\rm diag}\left(\mathbb{E}\left[P^k_{0}\right]\otimes \mathbb{E}\left[Q_1^{k}\right], \cdots, \mathbb{E}\left[P^k_{L-1}\right]\otimes \mathbb{E}\left[Q^k_l\right]\right).
$$

\begin{algorithm}[tb]
\caption{The GNTD Algorithm with K-FAC Method (GNTD-KFAC)}
\label{alg:GNTD-kfac}
\begin{algorithmic}
\STATE \textbf{Input:} Distribution $\mathcal{D}$, $Q$ function of neural network with parameters $\theta^0$, batch size $N$, damping rate $\omega$, momentum parameter $\eta$,  learning rate $\beta$.
\STATE Initialize the Q network parameters $\theta^0$.
\FOR{$k=0,1,\cdots,K-1$}
\STATE Obtain a dataset $\mathcal{D}_k=\{(s_i,a_i,r_i,s'_i,a'_i)\}_{i=1}^{N}$ of batch size $N$ from data distribution $\mathcal{D}$.
\FOR{$l=1,2,\cdots,L$}
\IF{$k=0$}
    \STATE Compute $\hat{P}^k_l = \frac{1}{N}\sum_{i=1}^N P^k_l(i)$ and $ \hat{Q}^k_l = \frac{1}{N}\sum_{i=1}^N Q^k_l(i)$.
\ELSE
    \STATE Compute $\hat{P}^{k-1/2}_l = \frac{1}{N}\sum_{i=1}^N P^k_l(i)$ and $ \hat{Q}^k_l = \frac{1}{N}\sum_{i=1}^N Q^k_l(i)$.
    \STATE Update $\hat{P}^{k}_l=(1-\eta)\hat{P}^{k-1}_l+\eta \hat{P}^{k-1/2}_l$.
\ENDIF
\STATE Update the parameters $\theta^{k+1}_l$ of each network layer  by \eqref{eq:kfac-direction} and \eqref{eq:kfac-update}. 
\ENDFOR
\ENDFOR
\STATE \textbf{Output:} $\theta^K$.
\end{algorithmic}
\end{algorithm}

After incorporating the identity matrix originated from the Levenberg-Marquardt method, then we approximately calculate the matrix inversion as follows  
$$
\begin{aligned}
    &(H(\theta^k)+\omega I)^{-1}\\
    \approx &\  \textbf{\rm diag}\left(\left(\mathbb{E}\left[P^k_0\right]\otimes \mathbb{E}\left[Q^k_1\right]+\omega I\right)^{-1}, \cdots,\left(\mathbb{E}\left[P^k_{L-1}\right]\otimes \mathbb{E}\left[Q^k_L\right]+\omega I\right)^{-1}\right) \\
    \approx &\  \textbf{\rm diag}\left(\left(\mathbb{E}\!\left[P^k_0\right]\!+\!\sqrt{\omega} I\right)^{\!-1}\!\otimes\! \left(\mathbb{E}\!\left[Q^k_1\right]\!+\sqrt{\omega} I\right)^{\!-1}\!, \cdots, \right. \\
    &\ \left.\left(\mathbb{E}\!\left[P^k_{L-1}\right]\!+\sqrt{\omega} I\right)^{\!-1}\!\otimes \left( \mathbb{E}\left[Q^k_L\right]\!+\sqrt{\omega} I\right)^{\!-1}\right).
\end{aligned}
$$
For the stochastic sampling case where the expectation are approximated by the sample averages, we let $\hat{P}^k_l$ and $\hat{Q}^k_l$ be the empirical estimators of $P^k_l$ and $Q^k_l$, which are given as 
\begin{eqnarray*}
\hat{P}^k_l = \frac{1}{N}\sum_{i=1}^N P^k_l(i)\qquad\mbox{and}\qquad \hat{Q}^k_l = \frac{1}{N}\sum_{i=1}^N Q^k_l(i),
\end{eqnarray*} 
where $P^k_l(i)$ and $Q^k_l(i)$ are constructed by running Algorithm \ref{alg:forward-backward-nn} for the $i$-th data point and utilize \eqref{eq:PQ}. Similarly, let $\hat{g}^k$ (c.f. \eqref{eq:stochastic}) be an estimator of the semi-gradient $g(\theta^k)$ (c.f. \eqref{eq:H-and-g}), and let $\hat{g}^k_l$ be the semi-gradient of the $l$-th layer. Then the descending direction for the $l$-th layer is
\begin{equation}
\label{eq:kfac-direction}
\begin{aligned}
    d^k_l\approx&\ (\hat{H}^k_l+\omega I)^{-1}\hat{g}^k_l
    \approx (\hat{P}^k_l+\sqrt{\omega} I)^{-1}\otimes(\hat{Q}^k_l+\sqrt{\omega} I)^{-1} \cdot \hat{g}^k_l \\
    =&\ \textbf{\rm Vec}\left((\hat{Q}^k_l+\sqrt{\omega} I)^{-1}\textbf{\rm Mat}(\hat{g}^k_l)(\hat{P}^k_l+\sqrt{\omega} I)^{-1}\right).
\end{aligned}
\end{equation}
Then we can naturally get the expression for the parameter update:
\begin{equation}
\label{eq:kfac-update}
\theta^{k+1}_{l}=\theta^k_l-\beta d^k_l, \quad\mbox{for}\quad l=1,2,\cdots,L.
\end{equation}
See Algorithm \ref{alg:GNTD-kfac} for details.

\section{Experiments}
\label{sect:ex}
Finally, we conduct a series of experiments over the OpenAI Gym \cite{brockman2016openai} tasks and demonstrate the efficiency of GNTD method under a variety of settings. For the sake of simplicity, we default to Algorithm \ref{alg:GNTD-kfac} (i.e. GNTD-KFAC) as the implemented version of GNTD in this section.

In details, we first examine the advantage of GNTD over TD in on-policy reinforcement learning setting, where the policy evaluation serves as a built-in module of the policy iteration method. Second, we also consider a few offline RL tasks, where we extend the proposed method to the Q-learning settings. All the compared learning algorithms are trained without exploration. We compare the performance of different algorithms in terms of the Bellman error and the final return.

\subsection{Policy Optimization with GNTD Method}
\label{section:pg-gntd}
First, we present the experiments where GNTD and TD are executed as built-in modules of an entropy regularized \cite{haarnoja2018soft/ac} policy iteration method. Typically, policy iteration is divided into two steps: policy evaluation and policy improvement. 
In details, given an initial policy $\pi_0$, our agents collect a data batch and then perform a 25-step policy evaluation to obtain $Q^{\pi_0}$, by either GNTD or TD method. Then, we take a 1-step policy gradient (PG) ascent to the entropy regularized total reward:
$$
\max_{\pi}\mathbb{E}_{s\sim\mathcal{D}, a\sim \pi(\cdot \mid s)}[Q^{\pi_0}(s,a)+\lambda H(\pi(\cdot\!\mid\! s))],
$$

where $H(\pi(\cdot\!\mid\! s))=\mathbb{E}_{a\sim \pi(\cdot\mid s)} [-\log \pi(a\!\mid\! s)]$ is the entropy of the density function $\pi(\cdot\!\mid\! s)$. Then the agent execute the new policy to collect a new data batch, and loop through policy optimization until convergence.

\begin{figure*}[htbp]
\vspace{-0.3cm} 
\begin{center}
\centering
\subfigure{
\includegraphics[width=0.4\textwidth,height=0.32\textwidth]{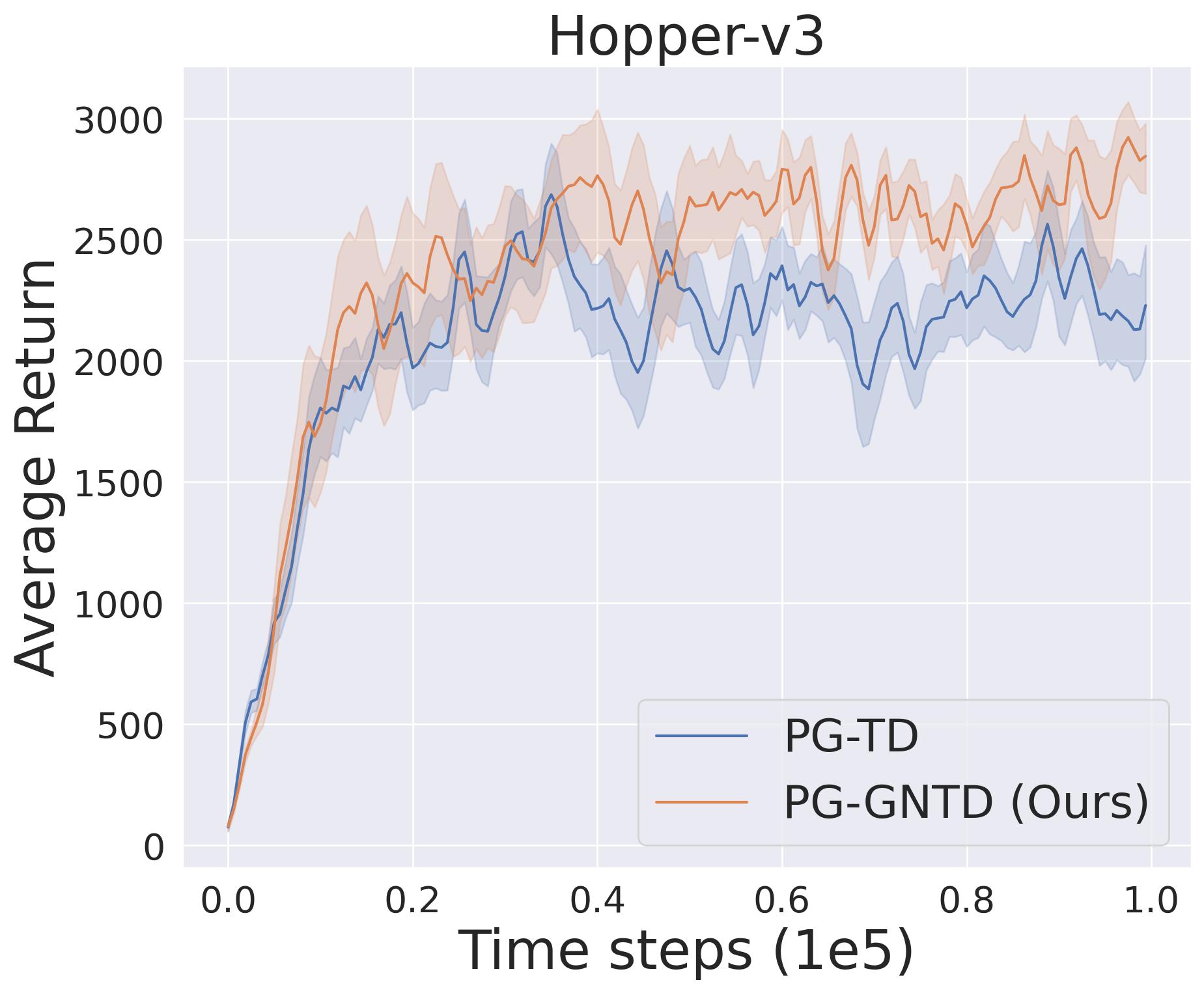}
}
\hspace{-3mm}
\subfigure{
\includegraphics[width=0.4\textwidth,height=0.32\textwidth]{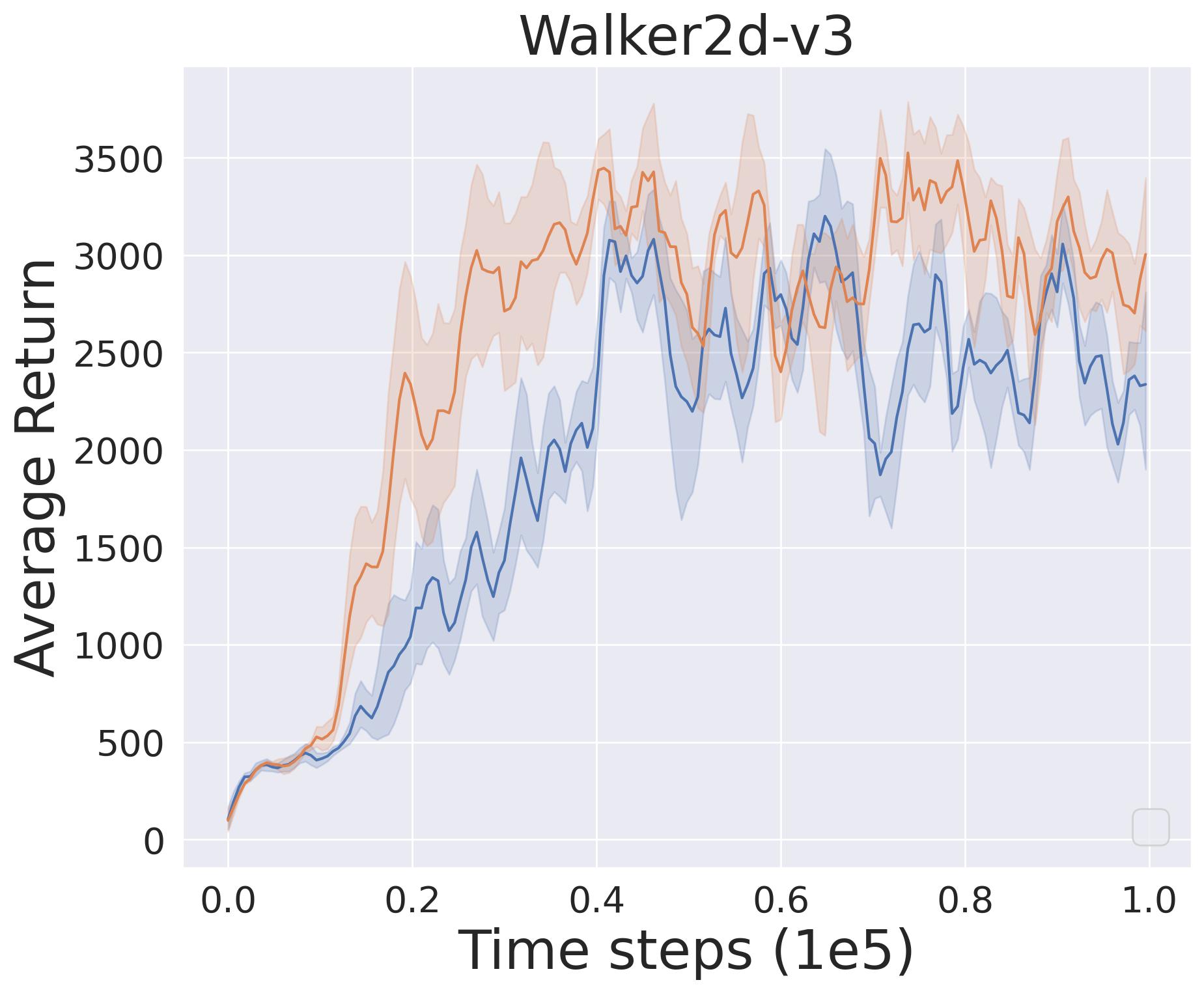}
}

\vskip -8pt
\subfigure{
\includegraphics[width=0.4\textwidth,height=0.32\textwidth]{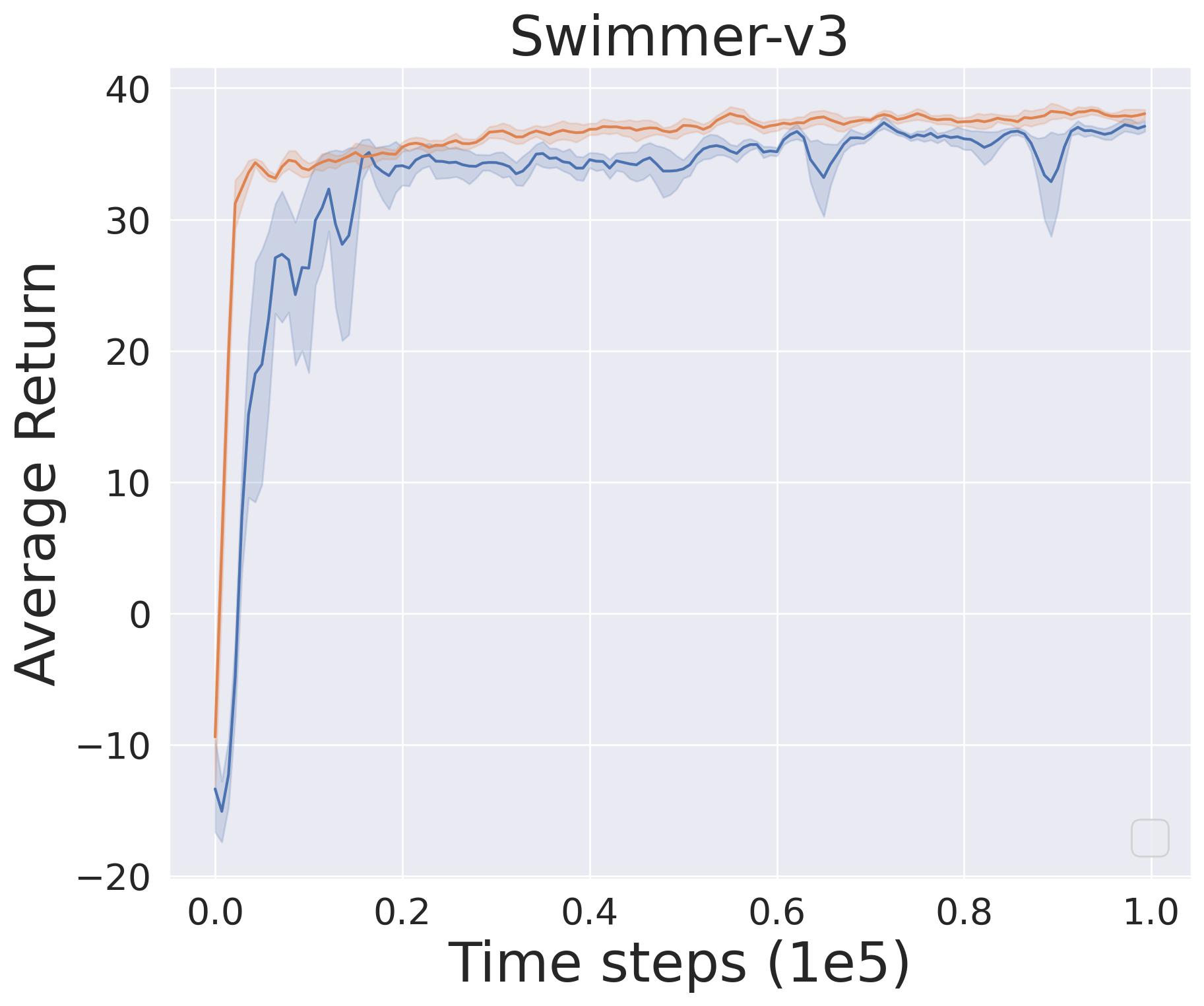}
}
\hspace{-3mm}
\subfigure{
\includegraphics[width=0.4\textwidth,height=0.32\textwidth]{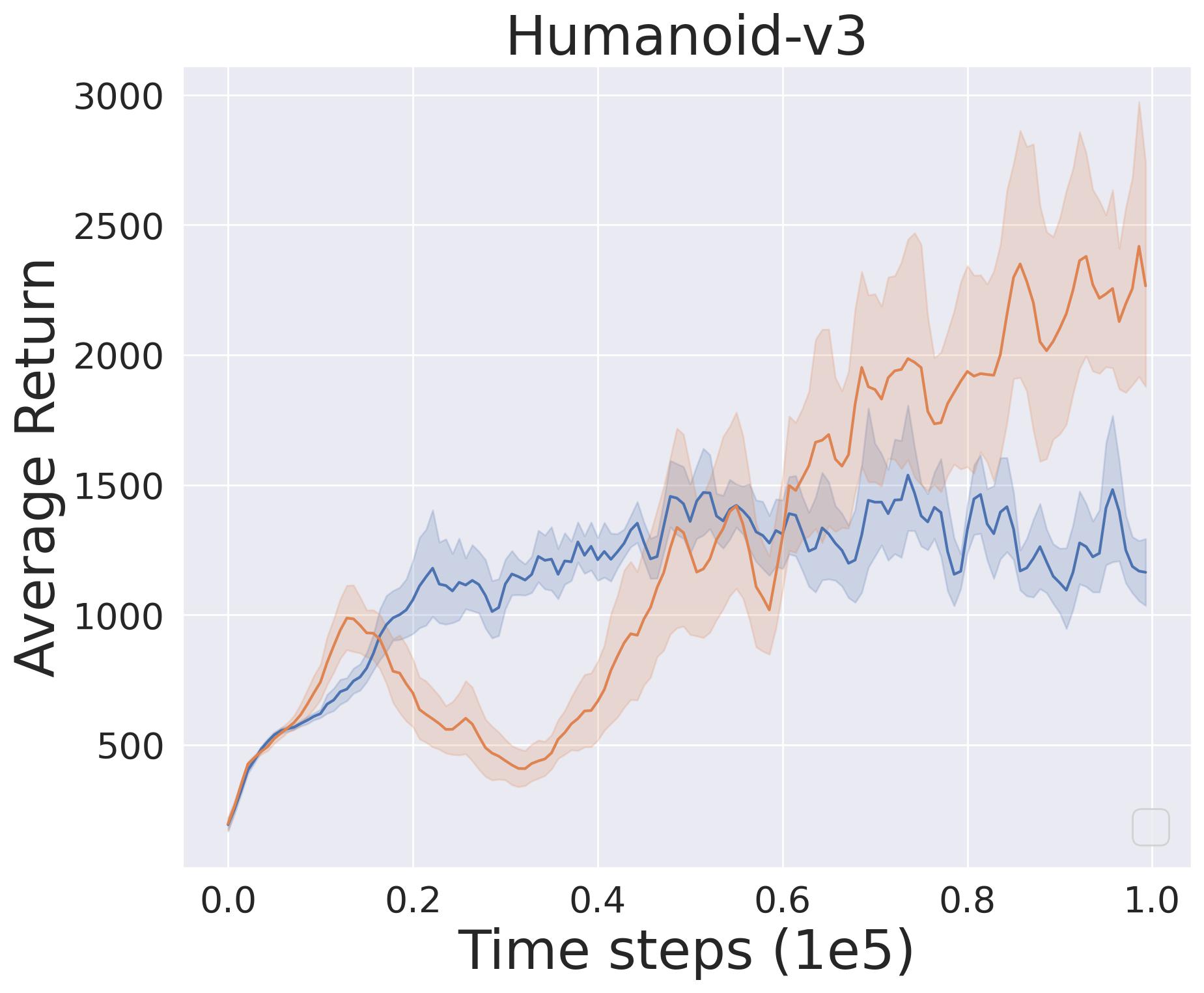}
}
\vspace{-0.3cm}
\caption{Training curves over 5 random seeds on OpenAI Gym MuJoCo tasks. We use TD and GNTD to handle policy evaluation in the policy gradient (PG) algorithm, respectively. The shaded area captures the standard deviation at each iteration. ``Time steps'' refers to the number of occurrences in which various algorithms utilize samples of the same batch size. }
\label{fig:pg-td-gntd}
\end{center} 
\end{figure*}

\vspace{-0.2cm}
\begin{table}[h]
\begin{center}
\tabcolsep=0.1cm
\begin{small}
\begin{tabular}{l|cccc|ccccr}
\toprule
\ & TD & DQN & GNTD & GNDQN & TD & DQN & GNTD & GNDQN \\
Data Set & \multicolumn{4}{c|}{\emph{Bellman Error}}& \multicolumn{4}{c}{\emph{Total Time} (s)} \\
\midrule
CartPole-rep    & 757.93 & 0.6 & 3.03 &\textbf{0.51}&   -- & 23.17 & 26.59 &\textbf{15.46}\\
CartPole-med-rep    & 97.78 & 0.66 & 1.51 & \textbf{0.58}&  -- & 36.65 & 30.73 & \textbf{26.96} \\
Acrobot-rep    &1.32 & 0.63 & 0.71 & \textbf{0.52}& --&103.55 & \textbf{56.21} & 99.98   \\
Acrobot-med-rep &1.41 & 0.68 & 0.68 & \textbf{0.58}& -- & 95.67 & \textbf{46.33} & 89.71 \\
\bottomrule
\end{tabular}
\vspace{0.3cm}
\caption{The left half of the table shows the average Bellman error of four algorithms TD, DQN, GNTD, and GNDQN on various datasets, while The right half of the table shows the average total time (s) required for these four algorithms to reach the corresponding reward threshold. Here CartPole has a threshold of $-400$, while Acrobot has a threshold of $-100$. }
\label{tab:bellman-error}
\end{small}
\end{center}
\vspace{-0.5cm}
\end{table}

For the policy function $\pi$ and state-action value function $Q$, we employ two layer neural networks. For computational efficiency, we implement the GNTD alorithm with K-FAC method. We set the damping rate $\omega=0.25$ and the learning rate $\beta=0.0003$ of the $Q$-function. 

Figure \ref{fig:pg-td-gntd} shows the experimental results under several
on-policy OpenAI gym environments. PG-GNTD and PG-TD refer to the policy optimization based on GNTD and TD, respectively. It can be observed that PG-GNTD converges faster than PG-TD in Hopper, Walker2d and Swimmer environments. It also fetches higher final rewards in all tasks.

\subsection{Offline Reinforcement Learning Tasks}
In this section, we present the experiments for both discrete and continuous offline RL tasks, where we will focus on optimizing rather than evaluating the policy. We compare the performance of our method against several benchmarks in terms of the Bellman error and average return.

\begin{figure*}[htbp]
\vspace{-0.3cm} 
\begin{center}
\centering
\subfigure{
\includegraphics[width=0.4\textwidth,height=0.32\textwidth]{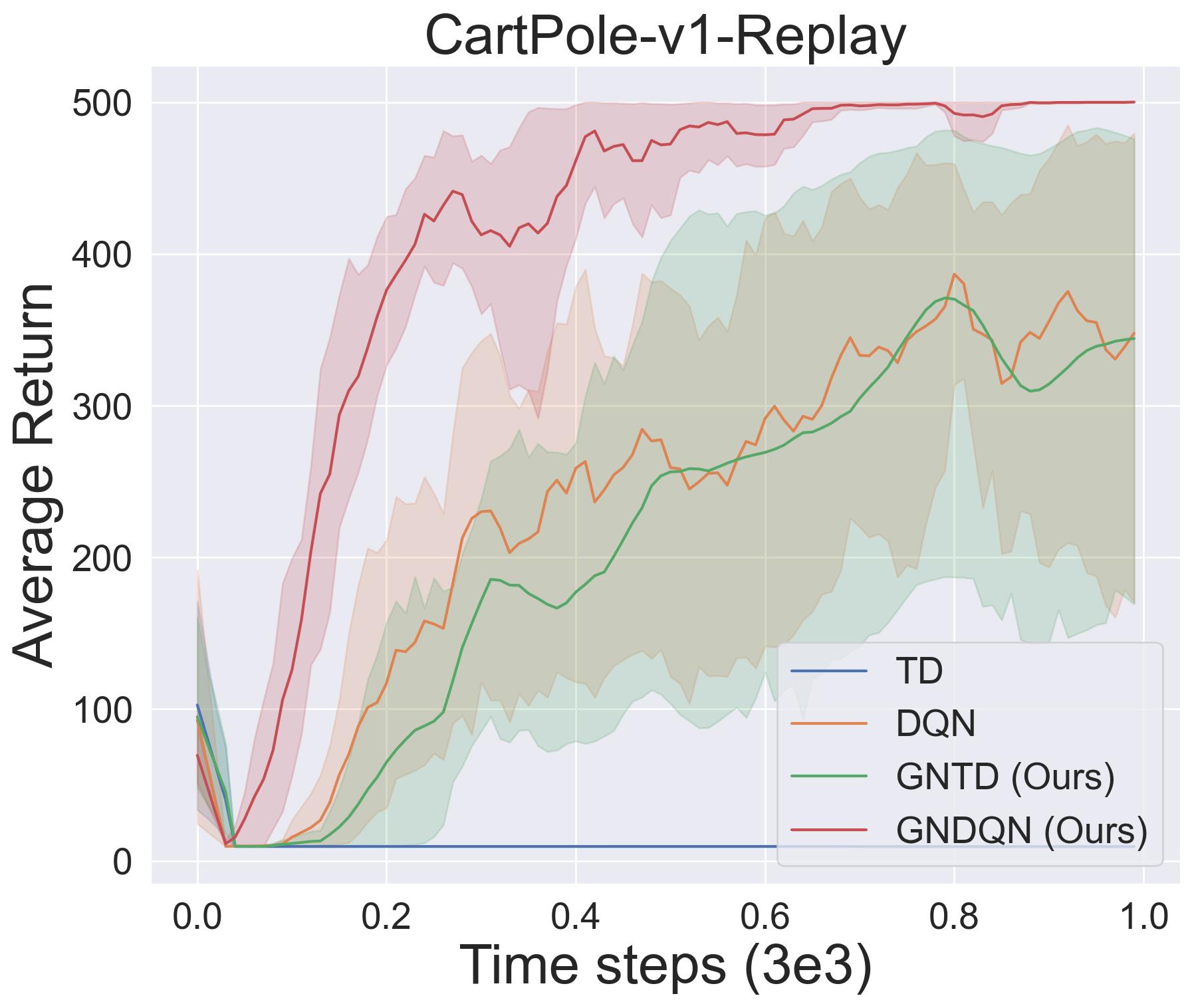}
}
\hspace{-3mm}
\subfigure{
\includegraphics[width=0.4\textwidth,height=0.32\textwidth]{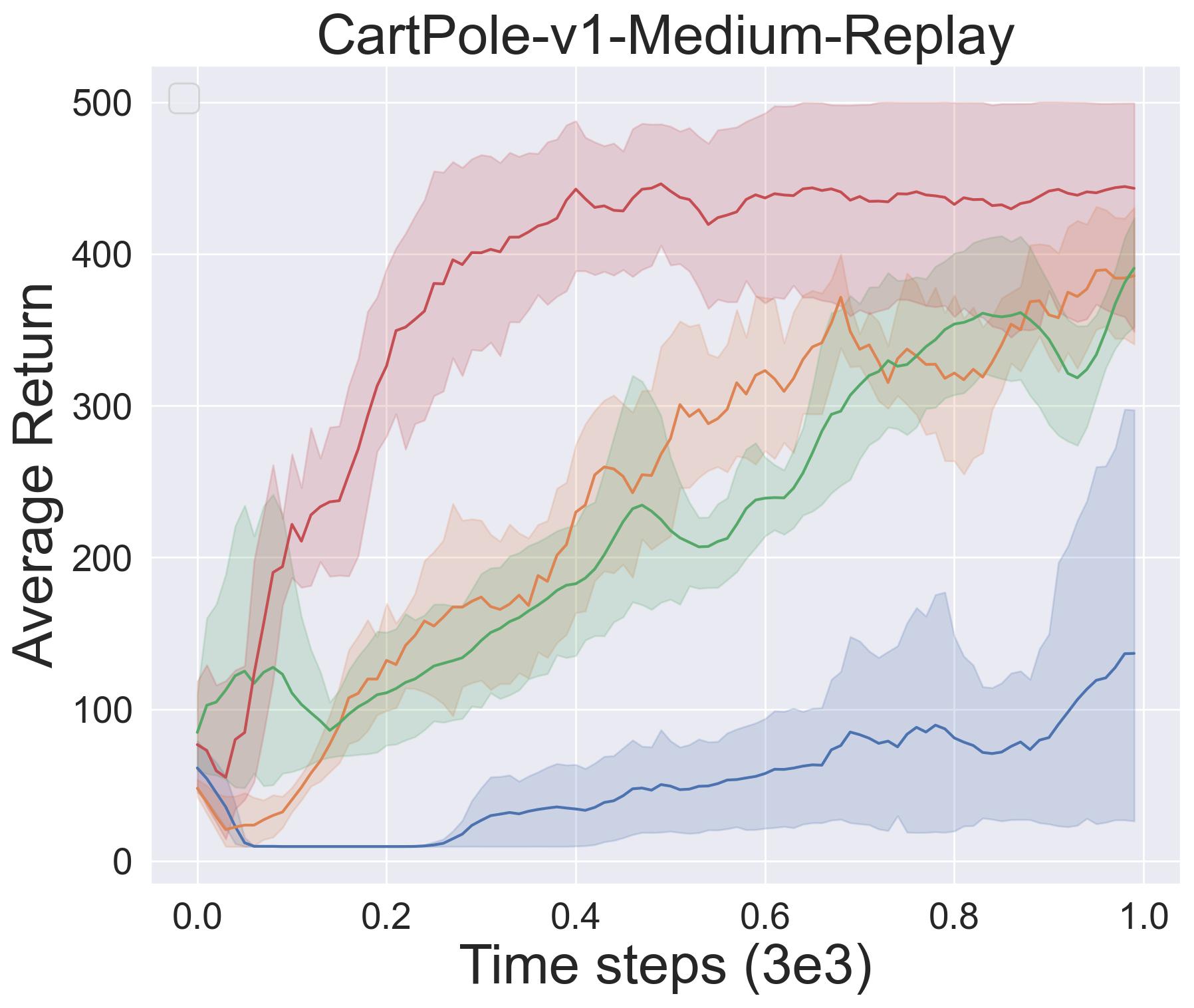}
}

\vskip -8pt
\subfigure{
\includegraphics[width=0.4\textwidth,height=0.32\textwidth]{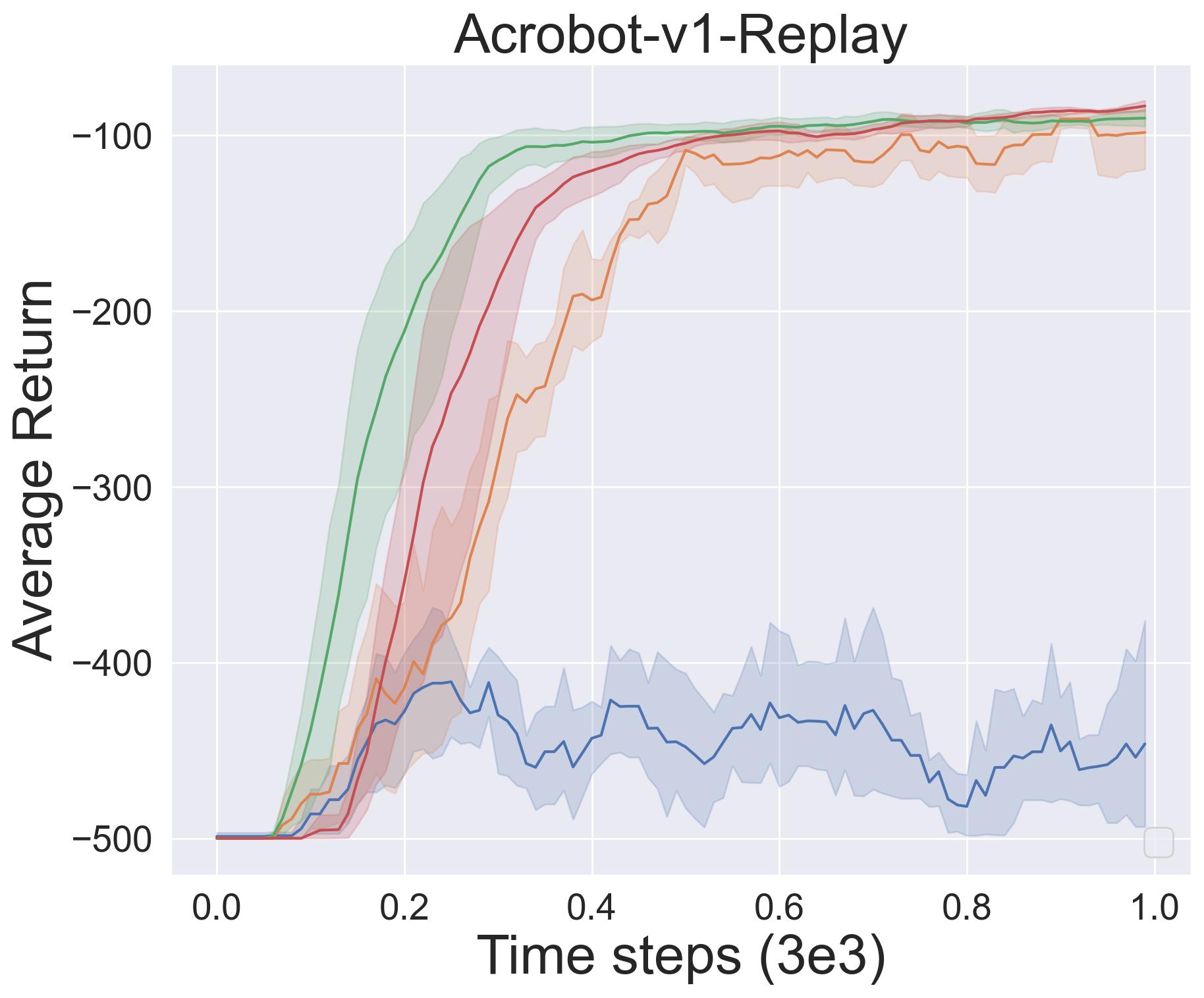}
}
\hspace{-3mm}
\subfigure{
\includegraphics[width=0.4\textwidth,height=0.32\textwidth]{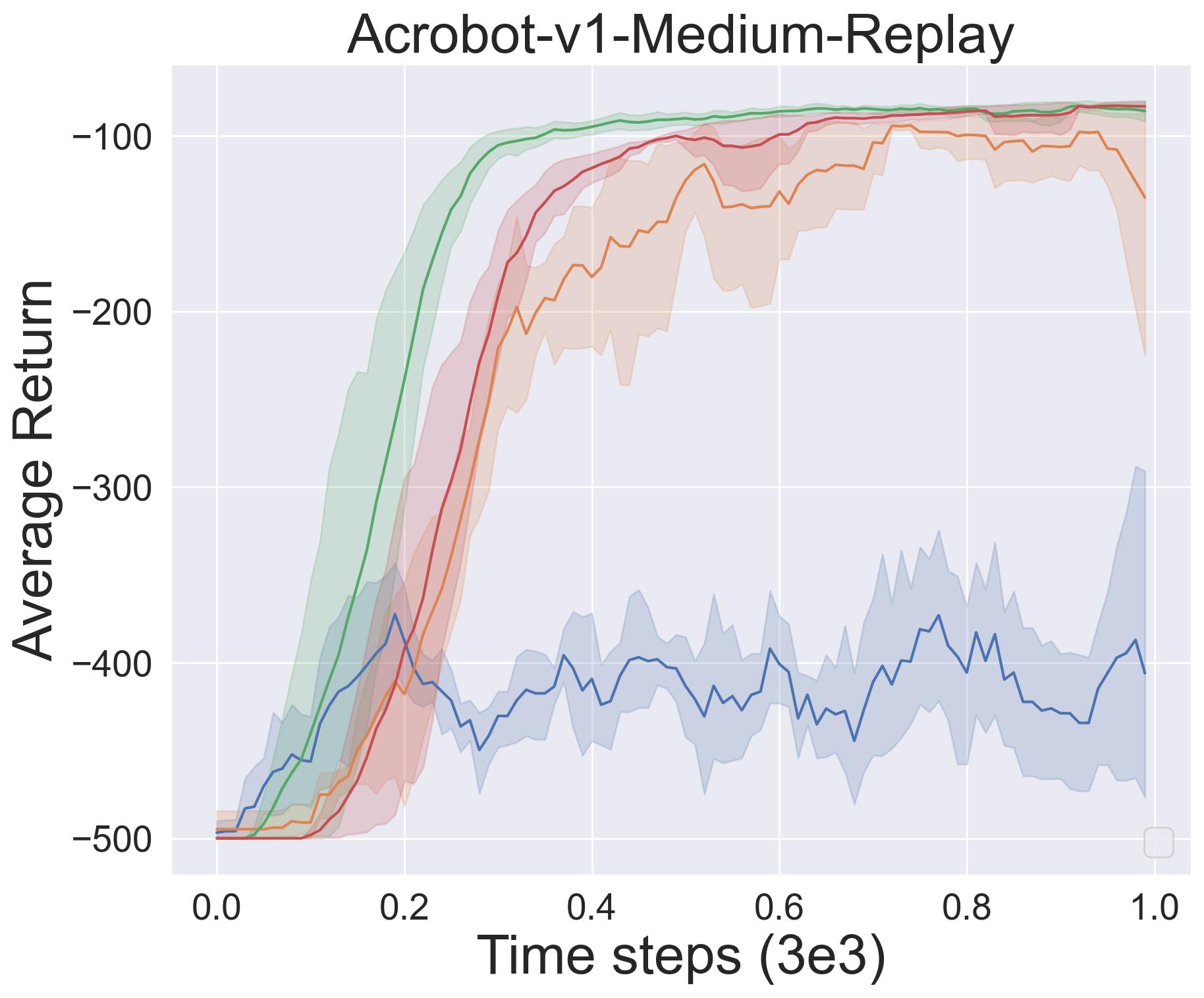}
}
\vspace{-0.3cm} 
\caption{Training curves over 5 random seeds on OpenAI Gym offline discrete tasks. The shaded area captures the standard deviation at each iteration.} 
\label{fig:q-learning-offline}
\end{center}
\vspace{-0.5    cm} 
\end{figure*}

\subsubsection{Discrete Action Tasks}
\label{section:gntd-dqn}
In this experiment, we present the experimental results under the OpenAI Gym CartPole-v1 and Acrobot-v1 environments. The tested algorithms includes TD, DQN \cite{mnih2013playing/dqn}, GNTD, and GNDQN. In particular, both TD and GNTD are adapted to the Q-learning setting where the Bellman operator is replaced with \emph{optimal} Bellman operator. 
GNDQN is a variant of GNTD method that incorporates a DQN-style momentum update to the target network, while taking Gauss-Newton steps with K-FAC method to update the weight matrices. Furthermore, the four algorithms use the same neural network architecture and has the same learning rate of $\beta=0.0003$. The size of all  offline datasets is chosen as $100000$, and 
we set the damping rate to be $\omega=0.25$.

Compared to the on-policy setting, offline RL requires strong conditions on the data distribution in order to obtain an optimal or near optimal policy. According to \cite{fan2020theoretical/fqi, agarwal2020/offRL/offQL}, we consider the following types of datasets: 

\begin{itemize}[leftmargin=*]
    \item \textbf{Replay datasets.} Train an online policy until convergence and use all samples during training.
    \item \textbf{Medium-replay datasets.} Train an online sub-optimal policy and use all samples during training.
\end{itemize}

From Figure \ref{fig:q-learning-offline} and Table \ref{tab:bellman-error}, it can be observed that GNTD outperforms TD in terms of both convergence speed, final reward, and Bellman error. After incorporating the momentum into the target network update, our GNDQN also dominates DQN in all the reported performance measures.

\subsubsection{Continuous Tasks}
\label{section:gntd3+bc}
Finally, we examine the performance of GNTD on the OpenAI Gym MuJoCo tasks using D4RL datasets. In these tasks, we propose a GNTD3+BC variant our method that merges GNTD with TD3+BC \cite{fujimoto2021minimalist/offRL}. In details, it is common to add a behavioral cloning regular term $\lambda$ to constrain the expected total return
\begin{equation*}
\label{eq:gntd-actor}
    J_\pi(\phi)= \mathbb{E}_{(s, a) \sim \mathcal{D}}\left[ Q(s, \pi_\phi (s);\theta_1)-\lambda(\pi_\phi (s)-a)^2\right].
\end{equation*}

For the critic part, we apply GNTD method (K-FAC style) to minimize the following MSBE of the Clipped Double DQN \cite{fujimoto2018addressing/ac}, that is,
\begin{equation*}
\label{eq:gntd-critic}
    J_V(\theta_i)=\mathbb{E}_{\xi \sim \mathcal{D}}\left[\left(Q(s,a;\theta_i)-Q_{targ}(s,a)\right)\right]^2,\ i\in\{1,2\},
\end{equation*}
where 
$$
Q_{targ}(s,a)=r(s,a)+\min\{Q(s^\prime,\pi(s^\prime);\theta_{1,targ}), Q(s^\prime,\pi(s^\prime);\theta_{2,targ})\}.
$$
We update the relevant parameters by optimizing $J_\pi$ and $J_V$.

Table \ref{tab:d4rl-score} shows the final numerical results for GNTD3+BC and other baselines, including BC (behavioral cloning), BCQ \cite{fujimoto2019policy/offRL}, CQL \cite{kumar2020conservative/offRL} and TD3+BC \cite{fujimoto2021minimalist/offRL}. Compared with TD3+BC, GNTD3+BC has higher final returns and lower variance in multiple environmental settings.

\begin{table*}[t]
\vspace{-0.1cm} 
\begin{center}
\tabcolsep=0.1cm
\begin{small}
\begin{tabular}{lcccc|cr}
\toprule
 & BC & BCQ & CQL\cite{kumar2020conservative/offRL} & TD3+BC & GNTD3+BC(Ours) \\
\midrule
Halfcheetah-m & 42.4$\pm$ 0.2 & 47.2$\pm$ 0.4 & 37.2 & 46.5$\pm$ 17.6 & \textbf{56.7$\pm$ 0.3} \\
Hopper-m & 30.1$\pm$ 0.3 & 34.0$\pm$ 3.8 & 44.2  & \textbf{100.2$\pm$ 0.2} & \textbf{100.5$\pm$ 0.3} \\
Walker2d-m & 12.6$\pm$ 3.1 & 53.3$\pm$ 9.1 & 57.5 &79.4$\pm$ 1.6 & \textbf{81.1$\pm$ 1.3} \\
\hline
Halfcheetah-m-r & 34.5$\pm$ 0.8 & 33.0$\pm$ 1.7 & 41.9 & \textbf{42.2$\pm$ 0.8} & \textbf{42.6$\pm$ 0.4} \\
Hopper-m-r & 20.0$\pm$ 3.1 & 28.6$\pm$ 1.1 & 28.6  & \textbf{31.8$\pm$ 2.0} & \textbf{32.4$\pm$ 1.3} \\
Walker2d-m-r & 8.1$\pm$ 1.3 & 11.5$\pm$ 1.3 & 15.8  & \textbf{23.9$\pm$ 1.6} & \textbf{24.8$\pm$ 2.1} \\
\hline
Halfcheetah-m-e & 70.6$\pm$ 7.1 & 84.4$\pm$ 4.5 & 27.1 & 90.9$\pm$ 3.6 & \textbf{98.2$\pm$ 3.3} \\
Hopper-m-e & 92.5$\pm$ 15.1 & 111.4$\pm$ 1.2 & 111.4  & \textbf{111.9$\pm$ 0.4} & \textbf{112.0$\pm$ 0.1} \\
Walker2d-m-e & 12.2$\pm$ 3.3 & 50.7$\pm$ 7.2 & 68.1  & 94.6$\pm$ 15.5 & \textbf{104.5$\pm$ 4.4} \\
\hline
Halfcheetah-e & 104.9$\pm$ 1.4 & 96.6$\pm$ 2.8 & 82.4& 103.9$\pm$ 1.7 & \textbf{107.6$\pm$ 0.6} \\
Hopper-e & 111.3$\pm$ 0.9 & 108.7$\pm$ 5.1 & 111.2 & \textbf{112.3$\pm$ 0.1} & \textbf{112.2$\pm$ 0.}5 \\
Walker2d-e & 58.1$\pm$ 9.1 & 92.6$\pm$ 5.0 & 103.8 & 105.2$\pm$ 1.8 & \textbf{107.6$\pm$ 1.5} \\
\hline
\specialrule{0em}{1pt}{1pt}
Total & 597.3$\pm$ 37.7 & 752.0$\pm$ 43.2 & 728.9 &  942.8$\pm$ 46.9 & \textbf{980.2$\pm$16.1}  \\
\bottomrule
\end{tabular}
\caption{Average normalized score over the final 10 evaluations and 5 seeds. Under the same configuration, we run the four algorithms of BC, BCQ, TD3+BC, GNTD3+BC, and also use the results reported by CQL and TD3+BC before. $\pm$ captures the standard deviation. ``m'', ``m-r'', ``m-e'', ``e'' respectively denote ``median'', ``median-replay'', ``median-expert'', ``expert''.}
\label{tab:d4rl-score}
\end{small}
\end{center}
\vspace{-0.5cm}
\end{table*}

\appendix
\section{Supporting Lemmas}
\label{sect:appendix}
In this section, we present the proof of the technical lemmas used in our main theorems.
\begin{lemma}
\label{lemma:vector-bernstein}
Let $\left\{S_i\right\}_{i=1}^N$ be a sequence of independent random vectors. Assume $\mathbb{E} S_i=\mathbf{0}$ and $\left\|S_i\right\| \leq \sigma, \,\forall i$, then
$$
\mathbb{P}\left(\bigg\|\frac{1}{N}\sum_{i=1}^N S_i\bigg\|>\frac{\rho\sigma}{\sqrt{N}}\right)\leq \exp\Big\{\!\!-\frac{\rho^2}{3}\Big\} ,\quad\mbox{for}\quad \forall \rho \geq 0.
$$
\end{lemma}
The proof of this result is straightforward. By
$\|S_i\|\leq \sigma$, we have $\mathbb{E}[\exp\{\|S_i\|^2/\sigma^2\}]\leq \exp\{1\}.$ Then applying Lemma 4.1 \cite{lan2020first/concentration} proves this lemma. 

\begin{lemma}
\label{lemma:matrix-bernstein}
Suppose that $\|\nabla Q(s,a;\theta)\|\leq L_1$ for any $(s,a), \theta$. Then for any $\eta \in(0,1)$ and any iteration $\theta^k$, we have $\|\hat{H}^k-H(\theta^k)\|_2\leq \eta$ w.p. $1-\delta$ as long as the sample size $|\mathcal{D}_k|\geq\frac{6L_1^4}{\eta^2}\log \frac{2d}{\delta}$.
\end{lemma}

\begin{proof}
First, recall the definition that 
$$
\hat{H}^k=\frac{1}{|\mathcal{D}_k|}\sum_{\xi\in \mathcal{D}_k} \nabla Q(s,a;\theta^k) \nabla Q(s,a;\theta^k)^{\top}
$$
and
$$
H(\theta^k)=\mathbb{E}_{\xi\sim\mathcal{D}}\left[ \nabla Q(s,a;\theta^k) \nabla Q(s,a;\theta^k)^{\top}\right].
$$
Define $S(\xi)=\dfrac{1}{|\mathcal{D}_k|}\left( \nabla Q(s,a;\theta) \nabla Q(s,a;\theta)^{\top}-J(\theta)^{\top} U J(\theta)\right)$. It holds that $\mathbb{E}[S(\xi)]=0,\|S(\xi)\|_2\leq \frac{2L_1^2}{|\mathcal{D}_k|}$ and  $\|E[S(\xi)S(\xi)^{\top}]\|\leq \frac{2L_1^4}{|\mathcal{D}_k|^2}.$
By Matrix Bernstein Inequality \cite{tropp2015introduction/concentration}, we have
$$
\mathbb{P}\bigg(\Big\|\sum_{\xi\in\mathcal{D}_k} S(\xi)\Big\|\geq \eta\bigg)\leq 2d\cdot \exp\left\{\frac{-3|\mathcal{D}_k|\eta^2}{4L_1^4  (3+\eta)}\right\}. 
$$
Choosing the batch size to satisfy $2d\cdot \exp\left\{\frac{-3|\mathcal{D}_k|\eta^2}{4L_1^4  (3+\eta)}\right\}\leq \delta$ proves this lemma.
\end{proof}


\bibliographystyle{siamplain}
\bibliography{article}
\end{document}